\documentclass[11pt,letterpaper]{amsart}

\usepackage{amsmath, amscd, amssymb}
\usepackage[frame,cmtip,arrow,matrix,line,graph,curve]{xy}
\usepackage{graphpap, color}
\usepackage[mathscr]{eucal}
\usepackage{cancel}
\usepackage{verbatim}
\usepackage{adjustbox}
\usepackage{tikz}
\usepackage{float}
\usepackage{booktabs}
\usepackage{multirow}
\usepackage{tabularx}
\usepackage{bbm}
\usepackage{stmaryrd}

\usepackage{hyperref}

\numberwithin{equation}{section}

\newtheorem{theorem}{Theorem}[section]
\newtheorem{corollary}[theorem]{Corollary}
\newtheorem{proposition}[theorem]{Proposition}
\newtheorem{lemma}[theorem]{Lemma}

\newtheorem{conjecture}[theorem]{Conjecture}

\theoremstyle{definition}
\newtheorem{definition}[theorem]{Definition}
\newtheorem{question}[theorem]{Question}
\newtheorem{example}[theorem]{Example}
\newtheorem{remark}[theorem]{Remark}

\newtheorem{problem}[theorem]{Problem}

\newcommand{\Z}{\mathbb{Z}}
\newcommand{\Q}{\mathbb{Q}}

\newcommand{\PP}{\mathbb{P}}

\def\QQ{\mathbb{Q}}

\def\ZZ{\mathbb{Z}}

\def\sL{{\mathscr L}}

\def\sT{\mathscr{T}}

\newcommand{\cal}{\mathcal}

\def\cM{{\cal M}}

\def\fp{\mathfrak{p}}
\def\fq{\mathfrak{q}}
\def\fQ{\mathfrak{Q}}

\def\fc{\mathfrak{c}}
\def\fd{\mathfrak{d}}

\def\tfp{\widetilde{\fp}}
\def\tfq{\widetilde{\fq}}

\def\hbar{\overline{h}}

\def\Mbar{\overline{\cM}}

\DeclareMathOperator{\Aut}{Aut}

\DeclareMathOperator{\Exp}{Exp}
\DeclareMathOperator{\Log}{Log}

\def\Rep{\mathrm{Rep}}

\def\Ind{\mathrm{Ind} }

\def\Stab{\mathrm{Stab} }
\def\dim{\mathrm{dim} }

\def\log{\mathrm{log} }

\def\ch{\mathrm{ch} }

\def\Inv{\mathrm{Inv} }
\def\mult{\mathrm{mult}}

\def\and{\quad{\rm and}\quad}
\def\lra{\longrightarrow }

\def\beq{\begin{equation}}
\def\eeq{\end{equation}}
\def\ben{\begin{enumerate}}
\def\een{\end{enumerate}}

\def\and{\quad\text{and}\quad}

\def\a{\alpha}
\def\d{\delta}

\def\sfs{\mathsf{s}}
\def\sfp{\mathsf{p}}
\def\sfh{\mathsf{h}}
\def\sfe{\mathsf{e}}

\def\symS{\mathbb{S}}
\def\Lab{\sL}

\title[Algorithm and log-concavity of representations on  $H^*(\Mbar_{0,n})$]
{Recursive algorithm and log-concavity of representations on the cohomology of $\Mbar_{0,n}$}
\date{2024.08.21}

\author{Jinwon Choi}
\address{Department of Mathematics and Research Institute of Natural Science, Sookmyung Women's University, Seoul 04310, Korea}
\email{jwchoi@sookmyung.ac.kr}

\author{Young-Hoon Kiem}
\address{School of Mathematics, Korea Institute for Advanced Study, 85 Hoegiro, Dongdaemun-gu, Seoul 02455, Korea}
\email{kiem@kias.re.kr}

\author{Donggun Lee}
\address{Center for Complex Geometry, Institute for Basic Science (IBS), 55 Expo-ro, Yuseong-gu, Daejeon 34126, Korea}
\email{dglee@ibs.re.kr}

\thanks{JC was partially supported by NRF grant 2018R1C1B6005600. YHK was partially supported by Korea NRF grant 2021R1F1A1046556. DL was supported by the Institute for Basic Science (IBS-R032-D1).}

\begin{document}

\begin{abstract}
We provide a programmable recursive algorithm for the $\symS_n$-representations on the cohomology of the moduli spaces $\Mbar_{0,n}$ of $n$-pointed stable curves of genus 0.
As an application, we find explicit inductive and asymptotic formulas for the invariant part $H^*(\Mbar_{0,n}/\symS_n)$ and prove that its Poincar\'e polynomial is asymptotically log-concave.  
Based on numerical computations with our algorithm, we further conjecture that the sequence $\{H^{2k}(\Mbar_{0,n})\}$ of $\symS_n$-modules  is equivariantly log-concave.
\end{abstract}

\maketitle

\tableofcontents

\section{Introduction}

The moduli space $\Mbar_{0,n}$ of $n$-pointed stable rational curves, introduced in \cite{Knu}, is a central object in algebraic geometry which has been much studied.
On $\Mbar_{0,n}$, the symmetric group $\symS_n$ acts by permuting the marked points and it is natural to consider the following.

\begin{problem}\label{11}
	Compute the $\symS_n$-representations on the rational cohomology of $\Mbar_{0,n}$.
\end{problem}

In \cite[Theorem 6.1]{CKL}, we found an explicit formula for the $\symS_n$-characters of $H^*(\Mbar_{0,n})$ as a sum over  weighted rooted trees. Using this, we could provide partial answers to the question asking whether $H^{2k}(\Mbar_{0,n})$ is a permutation representation or not for each $k$. 
However, for an actual computation of the cohomology $H^{2k}(\Mbar_{0,n})$ as an $\symS_n$-module, one has to deal with the combinatorics of weighted rooted trees which quickly causes difficulties as $n$ grows. 
In this paper, we provide a programmable recursive algorithm (Theorem \ref{thm:intro1}) for the $\symS_n$-characters of $H^*(\Mbar_{0,n})$ by investigating on the inductive structure of weighted rooted trees.  The resulting formula is much simpler than the previous one in \cite{CKL} and can be easily implemented by softwares like Mathematica. 

Recently, the property of log-concavity has been much studied as it often reveals deep underlying structures like Hodge modules as in the celebrated works of June Huh. 
In \cite{ACM}, Aluffi-Chen-Marcolli proved the asymptotic ultra-log-concavity for the Betti numbers of $\Mbar_{0,n}$. As the moduli space 
$$\Mbar_{0,n}/\symS_n$$
of rational curves with \emph{unordered} marked points is also an important space, 
it seems natural to ask the following question which is equally interesting and perhaps more difficult. 

\begin{question}\label{12}
Do the Betti numbers $H^{2k}(\Mbar_{0,n}/\symS_n)$ of $\Mbar_{0,n}/\symS_n$ form a log-concave sequence?
\end{question}

Since $H^*(\Mbar_{0,n}/\symS_n)$ is the $\symS_n$-invariant part in $H^*(\Mbar_{0,n})$,
we can compute the Betti numbers of $\Mbar_{0,n}/\symS_n$ by using our algorithm (Theorem \ref{thm:intro1}). In particular, we prove an inductive formula (Theorem \ref{thm:intro2}) and an asymptotic formula (Theorem \ref{thm:intro3}) for them. 
As a consequence, we find that the Poincar\'e polynomial of $\Mbar_{0,n}/\symS_n$ is asymptotically log-concave (Corollary \ref{cor:asymp.p_k}) but not ultra-log-concave (Corollary \ref{cor:ultra.lc}). 
More generally, based on our numerical computations, we conjecture that the multiplicities of irreducible representations of $\symS_n$ in $H^*(\Mbar_{0,n})$ form log-concave sequences (Conjecture \ref{conj:mult.lc})  and that the sequence $H^{2k}(\Mbar_{0,n})$ of $\symS_n$-modules is equivariantly log-concave (Conjecture \ref{conj:equiv.lc}).

\subsection{Representations on the cohomology of $\Mbar_{0,n}$}
For more detailed discussions, let us first recall previously known results about Problem \ref{11}. 

In \cite{Get}, Getzler first developed a method to compute the characters of the $\symS_n$-modules $H^{2k}(\Mbar_{0,n})$ by using the theory of modular operads. Later, Bergstr\"om and Minabe in \cite{BM} found a recursive algorithm for the characters which requires computations on Hassett's moduli spaces of  weighted pointed stable curves \cite{Has} which are numerous for each $n$.

In \cite{CKL}, we introduced a new approach for computing the $\symS_n$-representations on $H^{*}(\Mbar_{0,n})$. Based on the theory of $\d$-stable quasimaps developed in \cite{CK}, we related the $\symS_n$-representations on $H^*(\Mbar_{0,n})$ with those on $H^*(\Mbar_{0,n+1})$, where $\Mbar_{0,n+1}$ is equipped with an $\symS_n$-action that permutes the first $n$ markings while keeping the last fixed. The $\symS_n$-representations on $H^*(\Mbar_{0,n+1})$ can be computed by analyzing the Kapranov map $\Mbar_{0,n+1} \to \PP^{n-2}$, which can be factored into a sequence of $\symS_n$-equivariant blowups \cite{Kap}.

More precisely, the wall crossings of the moduli spaces of $\d$-stable quasimaps yield an $\symS_n$-equivariant factorization
\beq \label{5}\Mbar_{0,n+1}\cong \fQ^{\d=\infty}\lra \cdots \lra \fQ^{\d=0^+}\lra \Mbar_{0,n}\eeq
of the forgetful morphism $\Mbar_{0,n+1}\to \Mbar_{0,n}$ which forgets the last marking.
Moreover, every map in \eqref{5} is either a smooth blowup, a $\PP^1$-bundle, or a combination of both. By applying the blowup formula and the projective bundle formula, it was shown in \cite[\S4]{CKL} that the computation of the $\symS_n$-representations on the cohomology of $\Mbar_{0,n}$ is reduced to that of graded $\symS_n$-representation
\beq\label{2}\bigoplus_{k=0}^{n-2}H^{2k}(\Mbar_{0,n+1})t^k \; \in \mathrm{Rep}(\symS_n)[t].\eeq

The computation of \eqref{2} is more accessible, as one can repeatedly  
apply the blowup formula to the factorization \eqref{eq:Kap.map} of the Kapranov map.
This computation was carried out in \cite[\S5]{CKL}, yielding the formula
\beq\label{4}H^{2k}(\Mbar_{0,n+1})=\sum_{T\in \sT_{n,k}}U_T \;\in \mathrm{Rep}(\symS_n)\eeq
where $\sT_{n,k}$ denotes the set of \emph{weighted rooted trees} with $n$ inputs and weight $k$, and $U_T$ is the $\symS_n$-representation associated to $T$. By combining these results, we obtained a closed formula for the $\symS_n$-representations on $H^*(\Mbar_{0,n})$. See \S \ref{s:review} for further details about \cite{CKL}.

In principle, the formula in \cite{CKL} should lead to an explicit computation of $\symS_n$-representations on $H^*(\Mbar_{0,n})$. However, for actual numerical computations, one has to deal with the combinatorial complexity of weighted rooted trees, which increases rapidly as $n$ or $k$ grows.

\subsection{Recursive algorithm}
The first goal of this paper is to overcome the combinatorial difficulty of \eqref{4} and provide a \textbf{programmable} recursive algorithm for computing the $\symS_n$-representations on  $H^*(\Mbar_{0,n})$.

To state the precise results, let us consider the generating series of the characters of the representations as graded symmetric functions as follows.
\beq\label{10}\begin{split}
	&P:=1+\sfh_1 +\sfh_2 +\sum_{n\geq 3,k\geq 0}P_{n,k} t^k, \qquad P_{n,k}:=\ch_{\symS_n}\left(H^{2k}(\Mbar_{0,n})\right)\\
	&Q:=1+\sfh_1 +\sum_{n\geq 2,k\geq0}Q_{n,k} t^k, \qquad Q_{n,k}:=\ch_{\symS_n}\left(H^{2k}(\Mbar_{0,n+1})\right)
\end{split}\eeq
where $\ch_{\symS_n}(-)$ denotes the Frobenius characteristic map, and $\sfh_n$ denotes the $n$-th complete homogeneous symmetric function (see \S\ref{ss:Frob} for definitions).
By \eqref{4}, we have 
$$Q=1+\sfh_1 +\sum_{n\geq2,k\geq0}Q_{n,k} t^k, \qquad Q_{n,k}=\sum_{T\in \sT_{n,k}}\ch_{\symS_n}(U_T).$$
We define another generating function 
\beq\label{9} Q^+=\sfh_1 +\sum_{n\geq2,k>0}Q_{n,k}^+ t^k, \qquad Q_{n,k}^+:=\sum_{T\in \sT_{n,k}^+}\ch_{\symS_n}(U_T),\eeq
where $\sT_{n,k}^+\subset \sT_{n,k}$ denotes a subset consisting of elements having positive weight at the root vertex (see Definition~\ref{def:Tplus}). %
With these notations, we prove the following.

\begin{theorem}[Corollary \ref{cor:CKL.gen} and Theorem~\ref{thm:main1}] \label{thm:intro1} $\;$
\begin{enumerate}
\item $P$ and $Q$ are related by the formula
\beq\label{6}
(1+t)P=(1+t+\sfh_1 t)Q-t\sfs_{(1,1)}\circ Q,
\eeq
where $\circ$ denotes the plethysm (see \S\ref{ss:plethysm}) and $\sfs_{(1,1)}$ is the Frobenius characteristic of the sign representation of $\symS_2$ (see \S\ref{ss:Frob}).
	\item  $Q^+$ satisfies a recursive formula 
	\beq \label{8}Q^+=\sfh_1 +\sum_{r\geq 3}\left(\sum_{i=1}^{r-2}t^i\right)(\sfh_r\circ Q^+).\eeq
	\item  $Q$ is the \emph{plethystic exponential} (see \S\ref{ss:Exp}) of $Q^+$, that is,
	\beq\label{16}Q=\Exp(Q^+),\eeq
	where $\Exp(-):=\sum_{a\geq 0}\sfh_a\circ(-)$.
\end{enumerate}
\end{theorem}

By (1), $P$ follows from $Q$ and by (3), $Q$ follows from $Q^+$ while $Q^+$ can be directly computed by (2). 

The formula \eqref{6} is a consequence of \cite[Theorem 4.8]{CKL}, derived from $\d$-wall crossings in \eqref{5}. The formulas \eqref{8} and \eqref{16} stem from the inductive structure of weighted rooted trees. If we restrict to weighted rooted trees with positive weight at the root vertex, a weighted rooted tree can be constructed by combining smaller such rooted trees. This inductive structure results in \eqref{8}. Allowing the root vertex to have zero weight and utilizing the properties of the plethysm product, we arrive at \eqref{16}. It turns out that a combination of an exponential form of \eqref{8} and \eqref{16} is an equivariant generalization of Manin's characterization of the generating series $\varphi$ for the Poincar\'e polynomials of $\Mbar_{0,n+1}$ in \cite{Man} (see Remark~\ref{rem:Manin}). This type of characterization of $\varphi$ had been generalized to higher genus cases \cite{DN}, but not to an equivariant setting.

We note that the resulting inductive formulas in Theorem \ref{thm:intro1} are significantly simpler and more straightforward than the one derived in \cite{BM}. These formulas only involve the multiplications of the Schur functions and the plethysm product of the form $\sfh_r\circ (-)$ and $\sfs_{(1,1)}\circ (-)$. The former is governed by the well-established theory of Littlewood-Richardson rule, and the latter can be efficiently computed using the properties of the plethysm (cf. \S \ref{ss:plethysm}). We implemented the algorithm into a Mathematica program and computed $P$ and $Q$ for $n\le 25$  (cf. Remark \ref{rem:comp.rep}).

\subsection{Equivariant log-concavity} 
Based on our computational results for the $\symS_n$-representations on the cohomology of $\Mbar_{0,n}$ and $\Mbar_{0,n+1}$, we investigate their equivariant properties. Especially, we are interested in the equivariant log-concavity.

A sequence $\{a_k\}$ of integers is called \emph{log-concave} if $a_k^2\ge a_{k-1}a_{k+1}$ holds for all $k$. Recently, Aluffi-Chen-Marcolli in \cite{ACM} established asymptotic (ultra) log-concavity of the Poincar\'e polynomial of $\Mbar_{0,n}$, which means that the sequence of Betti numbers of $\Mbar_{0,n}$ is (ultra) log-concave when $n$ is sufficiently large. 

How do we generalize the log-concavity for a sequence of $\symS_n$-modules?  
\begin{definition} \cite{GPY, MMPR}
A sequence $\{V_k\}_k$ of finite-dimensional $\symS_n$-representations is said to be \emph{(resp. strongly) equivariantly log-concave} if
	\[V_{k-1}\otimes V_{k+1}\subset V_k\otimes V_k \qquad \text{(resp. }~V_i\otimes V_l \subset V_j\otimes V_k)\]
	for all $k$ (resp. for all $i\leq j\leq k\leq l$ with $i+l=j+k$).
\end{definition}

Based on numerical computations, we propose the following conjecture, which is a refinement of log-concavity of the Betti numbers of $\Mbar_{0,n}$. 

\begin{conjecture}\label{conj:equiv.lc}
	The sequence of  $\symS_n$-representations $\left\{H^{2k}(\Mbar_{0,n})\right\}_{0\leq k\leq n-3}$
	is strongly equivariantly log-concave.
\end{conjecture}

Another way to describe log-concavity of $\symS_n$-modules is to consider 
multiplicities of irreducible representations.  
Let $S^\lambda$ denote the irreducible $\symS_n$-representation corresponding to a partition $\lambda$ of $n$.

\begin{conjecture}\label{conj:mult.lc}
For each partition $\lambda$ of $n$, the multiplicities of $S^\lambda$ in 
$\{H^{2k}(\Mbar_{0,n})\}_k$  form a log-concave sequence. The same holds for the multiplicities of $S^\lambda$ in $\{H^{2k}(\Mbar_{0,n+1})\}_k$.
\end{conjecture}

We checked that Conjectures \ref{conj:equiv.lc} and \ref{conj:mult.lc} hold for $n\le 25$. Although there is no direct implication between these two conjectures, Conjecture \ref{conj:mult.lc} can provide a supporting evidence for Conjecture \ref{conj:equiv.lc}. For example, one can show that if Conjecture \ref{conj:mult.lc} holds, the multiplicity of the trivial representation in $H^{2(k-1)}(\Mbar_{0,n})\otimes H^{2(k+1)}(\Mbar_{0,n}) $ is less than or equal to that in $H^{2k}(\Mbar_{0,n})^{\otimes 2}$.

\subsection{Betti numbers of the moduli space of rational curves with unordered marked points} 
Concerning Conjecture \ref{conj:mult.lc}, of particular interest are the multiplicities of  
the trivial representation or the dimensions of the invariant parts, because they are the Betti numbers of the moduli space $\Mbar_{0,n}/\symS_n$ of stable rational curves with unordered marked points.   

Consider the invariant analogues of \eqref{10} as follows.
\beq\label{eq:gen.series.inv}\begin{split}
	\fp:=1+q&+q^2+\sum_{n\geq3,k\geq0}\fp_{n,k}q^nt^k, \qquad \fp_{n,k}:=\dim H^{2k}(\Mbar_{0,n})^{\symS_n},\\
	\fq:=1&+q+\sum_{n\geq2,k\geq0}\fq_{n,k}q^nt^k,\qquad \fq_{n,k}:=\dim H^{2k}(\Mbar_{0,n+1})^{\symS_n},\\ 	
	\fq^+:&=q+\sum_{n\geq 2,k\geq0}\fq^+_{n,k} q^nt^k, \qquad  \fq^+_{n,k}:=\lvert \sT^+_{n,k}\rvert,
\end{split}
\eeq
where $q$ is a formal parameter recording $n$.
Note that 
\[\fp_n:=\sum_{k=0}^{n-3}\fp_{n,k} t^k \and \fq_n:=\sum_{k=0}^{n-2}\fq_{n,k}t^k \]
are the Poincar\'{e} polynomials of $\Mbar_{0,n}/\symS_n$ and $\Mbar_{0,n+1}/\symS_n$ respectively.

\begin{theorem}[Theorem~\ref{thm:main2} and Theorem \ref{thm:qwc.inv}]\label{thm:intro2}
	The generating series $\fp$, $\fq$ and $\fq^+$ satisfy the following identities.
\begin{enumerate}
\item $(1+t)\fp=(1+t+qt)\fq-\frac{1}{2}t(\fq^2-\fq^{[2]}),$
where $\fq^{[2]}(q,t):=\fq(q^2,t^2)$.
\item $\fq^+=q+\sum_{r\geq 3}\left(\sum_{i=1}^{r-2}t^i\right)\left(\sfh_r\circ \fq^+\right)$.
\item $ \fq=\Exp(\fq^+)$.
\end{enumerate}
Here, the operations $\circ$ and $\Exp$ are the plethysm and the plethystic exponential on $\Z\llbracket q,t\rrbracket$ respectively (see \S\ref{ss:function.A}).
\end{theorem}

To prove this theorem, we show that the projection to the invariant part
 is a $\ZZ$-algebra homomorphism which commutes with the plethystic exponential (Proposition \ref{prop:Inv.plethysm}). Then Theorem \ref{thm:intro2} can be extracted from Theorem \ref{thm:intro1}.

Note that Theorem \ref{thm:intro2} enables us to recursively compute $\fp_{n}$ and $\fq_{n}$,   
without computing the whole representations $P_{n,k}$ or $Q_{n,k}$.
Using this, we  
checked that they are log-concave polynomials for $n\leq 45$ (Remark~\ref{rem:comp.inv}).  
This is a supporting evidence for the following special case of Conjecture \ref{conj:mult.lc}. 
\begin{conjecture}\label{conj}
	The Poincar\'e polynomials of $\Mbar_{0,n}/\symS_n$ and $\Mbar_{0,n+1}/\symS_n$ satisfy the log-concavity, that is,
	\beq\label{15}\fp_{n,k}^2\geq \fp_{n,k-1}\fp_{n,k+1} \and \fq_{n,k}^2\geq \fq_{n,k-1}\fq_{n,k+1}\eeq
hold for all $n$ and $k$.
\end{conjecture}

By considering $\fp_{n,k}$ and $\fq_{n,k}$ as functions of $n$, we can prove the following asymptotic formulas by Theorem \ref{thm:intro2} and \eqref{4}.

\begin{theorem}[Theorems~\ref{thm:asymp.q_k} and \ref{thm:asymp.p_k}]\label{thm:intro3}
	For $k\geq0$, we have 
	\[\fp_{n,k}= \frac{(k+1)^{k-2}}{(k!)^2}n^k+o(n^k) \and \fq_{n,k}= \frac{(k+1)^{k-1}}{(k!)^2}n^k+o(n^k).\]
\end{theorem}

From \eqref{4} and the fact that each $U_T$ contains the trivial representation with multiplicity one in its decomposition into irreducible representations, one can see that $\fq_{n,k}$ is exactly the number of weighted rooted trees with $n$ inputs and weight $k$. From this, it can be shown that $\fq_{n,k}$ grows at a polynomial rate of order $k$ as $n$ increases, and that the highest-order coefficient is given by the enumeration of the weighted rooted trees having weight one at each non-root vertex. Combinatorially, such weighted rooted trees can be related to labeled trees on $k+1$ vertices, the number of which is $(k+1)^{k-1}$ according to Cayley's formula. This number appears as the numerator in the formula for $\fq_{n,k}$. The formula for $\fp_{n,k}$ is then obtained by Theorem \ref{thm:intro2} (1).

With Theorem \ref{thm:intro3}, one can easily check that Conjecture \ref{conj} holds asymptotically (Corollary \ref{cor:asymp.p_k}). This provides  another supporting evidence for Conjecture \ref{conj} and also for Conjecture \ref{conj:mult.lc}. Interestingly, $\fp_n$ and $\fq_n$ are not ultra-log-concave (Corollary \ref{cor:ultra.lc}). 
Finally, we remark that our approach recovers \cite[Theorem 1.3]{ACM} (cf. Remark~\ref{rem:ACM}).

\subsection{The layout}

This paper is organized as follows. In \S \ref{sec:symftn}, we summarize necessary facts on $\symS_n$-representations and symmetric functions. In \S \ref{s:review}, we give a brief overview of the algorithm to compute $P$ and $Q$ developed in \cite{CKL}. In \S \ref{s:recursion}, we present the recursive formula for $Q$. In \S \ref{s:invariant}, we prove the recursive formulas for $\fp_n$ and $\fq_n$ of the multiplicities of the trivial representations. In \S \ref{s:asymptotic}, we study the asymptotic log-concavity of $\fp_n$ and $\fq_n$.
\medskip

All cohomology groups are singular cohomology with rational coefficients and all $\symS_n$-representations are over $\QQ$.

\medskip

\noindent \textbf{Acknowledgement.}
We thank Graham Denham, Tao Gui, June Huh, Shiyue Li, Jacob Matherne and Botong Wang for useful discussions.

\bigskip

\section{Preliminaries: $\symS_n$-representations and symmetric functions} \label{sec:symftn}
We review basics of the theory of representations of the symmetric groups and symmetric functions. For references, see \cite{Mac}, \cite{FH} and \cite{GeKa,Get}.
\subsection{Permutation representations}
Let $G$ be a finite group. Let $\Rep(G)$ be the free abelian group generated by irreducible representations of $G$.
\begin{definition}
	Let $H$ be a subgroup of $G$, and let $V$ be a representation of $H$.
	The \emph{induced $G$-representation} 
	of $V$ is defined to be the $G$-representation
\[\Ind^G_HV:=\bigoplus_{i}g_iV\]
where $\{g_i\}_i\subset G$ is the set of representatives of the left cosets in $G/H$, and the $G$-action on $\Ind^G_HV$ is defined as follows: for $g\in G $ and $w=g_iv\in g_iV$,
\[gw:=g_j(hv)\in g_jV\]
where $j$ and $h$ are uniquely determined by $gg_i=g_jh$ in $G$.
\end{definition}

The class of $\Ind^G_H V$ in $\Rep(G)$ does not depend on the choice of representatives $\{g_i\}$. Moreover, one can check that $\Ind^G_HV=\Ind^G_{H'} V$ in $\Rep(G)$ if $H$ and $H'$ are conjugate to each other.

\begin{definition}Let $G$ be a finite group. 
A representation $V$ of $G$ is called a \emph{permutation representation} of $G$ (\emph{spanned by $B$}) if there exists a basis $B$ of $V$ such that $B$ is invariant as a set under the $G$-action on $V$. If the basis $B$ consists of a single orbit of $G$, then $V$ is said to be \emph{transitive}.
	
\end{definition}

Every finite dimensional permutation representation $V$ of $G$ can be decomposed into a direct sum of transitive permutation subrepresentations 
\[V=\bigoplus_{i=1}^dV_i, \qquad \dim V_i^G=1\]
where $d=\dim V^G$ is the dimension of the $G$-invariant subspace of $V$. This decomposition is induced by the orbit decomposition of a $G$-invariant basis of $V$ and hence may not be unique. 

\begin{definition}\label{def:perm.rep}
	Let $H$ be a subgroup of $G$.
	We define the \emph{permutation representation of $G$ associated to $H$} to be the induced $G$-representation
	\[U_H:=\Ind^G_H\mathbbm{1}\]
	of the trivial representation $\mathbbm{1}$ of $H$.
\end{definition}
Note that $U_H$ is transitive, since it is naturally isomorphic to $\bigoplus_{\sigma\in G/H}\Q\cdot \sigma$ with natural $G$-action on it. Conversely, if the $G$-representation $V$ has an invariant basis $B$ consisting of a single $G$-orbit, then as a $G$-set, $B$ is isomorphic to $G/H$ for some $H\subset G$. Hence, $V= U_H$.

\begin{lemma}\label{lem:Schur.pos.perm.rep}
	Let $H,K$ be subgroups of $G$ with $H\subset K$. Then $U_K\subset U_H$ as $G$-representations.
\end{lemma}
\begin{proof}
	This is immediate since $U_H$ and $U_K$ are the permutation representations of $G$ spanned by the left cosets $G/H$ and $G/K$ respectively, and there is a $G$-equivariant surjection $G/H\to G/K$.
\end{proof}

\subsection{Representations of symmetric groups}\label{ss:symrep}
Let $R_n:=\Rep(\symS_n)$ be the free abelian group generated by irreducible representations of $\symS_n$ over $\QQ$.
We set $R(\symS_0):=\Z$. Consider the graded ring
\[R:=\prod_{n\geq0}R_n\]
where the ring structure is given by
\[V.W:=\Ind^{\symS_{m+n}}_{\symS_m\times \symS_n}V\otimes W \in R_{n+m}\]
for $V\in R_n$ and $W\in R_m$.

\begin{definition}\label{def:partition}
For an integer $n>0$, a \emph{partition} of $n$ is a nonincreasing sequence (or a multiset) $\lambda=(\lambda_1,\cdots,\lambda_\ell)$ of positive integers $\lambda_i$ with $\sum_i\lambda_i=n$. The nonzero $\lambda_i$ are called the \emph{parts} of $\lambda$. The number of parts is the \emph{length} of $\lambda$ and denoted by $\ell(\lambda)$. The notation $\lambda \vdash n$ means that $\lambda$ is a partition of $n$.

We sometimes write a partition in the form $\lambda=(\lambda_1^{\ell_1},\cdots, \lambda_m^{\ell_m})$ when $\lambda_i$ have the multiplicities $\ell_i$, and $\lambda_i$ are distinct. In this case, $\ell(\lambda)=\sum_{i=1}^m\ell_i$.

	The \emph{permutation module} $M^{\lambda}$ of $\symS_n$ associated to $\lambda$ is defined to be the permutation representation in Definition~\ref{def:perm.rep}
	\[M^\lambda:=U_{\symS_\lambda}=\Ind^{\symS_n}_{\symS_{\lambda}}\mathbbm{1}\]
	associated to the Young subgroup $\symS_\lambda:=\prod_{i=1}^\ell \symS_{\lambda_i}$  of $\symS_n$ associated to $\lambda$.
	
	For $n=0$, we consider $\lambda=(0)$ to be the unique partition of 0, and set $M^{(0)}:=1$ in $R_0=\Z$.

\end{definition}

It is well known that
$\{M^\lambda\}_{\lambda\vdash n}$ forms a free basis of $R_n$ for $n\geq 0$. So,
\[R=\ZZ\llbracket M^{(1)},M^{(2)},\cdots~ \rrbracket\]
is the power series ring in the trivial representations $\{M^{(n)}\}_{n\geq 1}$.

For each $\lambda\vdash n$, there exists an irreducible $\symS_n$-representation $S^\lambda$, called a \emph{Specht module}.
Each $S^\lambda$ is irreducible and $\{S^\lambda\}_{\lambda\vdash n}$ is the complete set of irreducible $\symS_n$-representations.
The Specht module $S^\lambda$ is a subrepresentation of $M^\lambda$, appearing with multiplicity one in the decomposition of $M^\lambda$ into irreducible representations.

For example, when $\lambda =(n)$, $M^{(n)}=S^{(n)}$ is the trivial representation, and when $\lambda=(1,\cdots,1)=(1^n)$,  $M^{(1^n)}$ is the regular representation
\beq \label{eq:reg.rep} M^{(1^n)}=\sum_{\lambda \vdash n}(\dim S^\lambda)\cdot S^\lambda\eeq
and $S^{(1^n)}$ is
the sign representation of $\symS_n$. For more details, see \cite{FH}.

\subsection{Frobenius characteristic map} \label{ss:Frob}

We denote by
\[\Lambda:=\lim_{\longleftarrow}\Z\llbracket x_1,\cdots,x_n\rrbracket ^{\symS_n} \]
the ring of symmetric functions in variables $\{x_i\}_{i\geq1}$, where $(-)^{\symS_n}$ denotes the invariant subring under the action of $\symS_n$ permuting $x_i$.
Let $\Lambda_n\subset \Lambda$ denote the subgroup consisting of elements of (total) degree $n$ in $x_i$.

	For $n\geq 1$, we denote by $\sfh_n$ the $n$-th complete homogeneous symmetric function and by $\sfp_n$ the $n$-th power sum symmetric function
	\[\sfh_n:=\sum_{1\leq i_1\leq \cdots\leq i_n}x_{i_1}\cdots x_{i_n} \quad \text{and} \quad \sfp_n:=\sum_{i\geq1}x_i^n,\]
and we set $\sfh_0=1$ for convenience. For a partition $\lambda=(\lambda_1,\cdots,\lambda_\ell)$ of $n$, we write
\[\sfh_{\lambda}:=\sfh_{\lambda_1}\cdots \sfh_{\lambda_\ell} \and \sfp_{\lambda}:=\sfp_{\lambda_1}\cdots \sfp_{\lambda_\ell}.\]

The symmetric functions $\sfh_n$ and $\sfp_n$ are related by the identity
\[\sum_{n\geq0}\sfh_n=\exp\sum_{n\geq1}\frac{\sfp_n}{n},\]
where $\exp$ denotes the usual exponential function. Equivalently, we can express $\sfh_n$ as a linear combination of the $\sfp_\lambda$. For a partition $\lambda$ of $n$, define
\[ z_\lambda = \prod_{i\ge 1} i^{m_i}\cdot m_i !, \]
where $m_i$ is the number of parts of $\lambda$ equal to $i$. Then we have
\[ \sfh_n = \sum_{\lambda\vdash n} z_\lambda^{-1} \sfp_\lambda. \]

It is well known that $\{\sfh_\lambda\}_{\lambda\vdash n}$ and $\{\sfp_\lambda\}_{\lambda\vdash n}$ form bases of $\Lambda_n$ and $\Lambda_n\otimes_\Z\Q$ respectively, for all $n$. In particular,
\[\Lambda=\ZZ\llbracket \sfh_1,\sfh_2,\cdots~\rrbracket  \and \Lambda\otimes_\Z\Q=\QQ\llbracket \sfp_1,\sfp_2,\cdots~\rrbracket \]
are power series rings in $\sfh_n$ and $\sfp_n$ respectively (see \cite[\S I.2]{Mac} or \cite[\S5.1]{Get}).

The $n$-th \emph{Frobenius characteristic} map $\ch_{\symS_n}:R_n\to \Lambda_n$ is defined by
\[\ch_{\symS_n}(V)=\frac{1}{n!}\sum_{\sigma \in \symS_n}\mathrm{Tr}_V(\sigma)\sfp_\sigma\]
for $V\in R_n$ and $\sfp_\sigma:=\sfp_{\lambda(\sigma)}$, where $\mathrm{Tr}_V(\sigma)$ denotes the trace of the linear endomorphism of $V$ defined by $\sigma$, and $\lambda(\sigma)$ denotes the partition of $n$ corresponding to the cycle type of the permutation $\sigma$. We set $\ch_0$ to be the identity on $R_0=\Lambda_0=\Z$. One can check that
\beq \label{eq:ch.M.to.h}\ch_{\symS_n}(M^\lambda)=\sfh_\lambda \quad  \text{ for every }\lambda\vdash n.\eeq
In particular, $\ch_{\symS_n}(R_n)\subset \Lambda_n$ so that $\ch_{\symS_n}$ is well defined in integer coefficients. It is well known that $\ch_{\symS_n}$ is an isomorphism (\cite[(I.7.3)]{Mac}).

The \emph{Schur functions} $\sfs_\lambda$ associated to $\lambda\vdash n$ are the images of $S^\lambda\in R_n$ under Frobenius characteristic map, that is, $\ch_{\symS_n}(S^\lambda)=\sfs_\lambda$. In particular, $\sfs_{(n)}= \sfh_n$. The Schur functions form a basis for $\Lambda_n$, as $S^\lambda$ do for $R_n$.

An element of $\Lambda$ is said to be \emph{Schur-positive} if it is a linear combination of $\sfs_\lambda$ with nonnegative coefficients, or equivalently, if it is the image under $\ch$ of the sum of genuine representations in $R$. 

\medskip

Let $R\llbracket t \rrbracket$ and $\Lambda\llbracket t \rrbracket$ be the power series rings in a formal variable $t$ with coefficient rings $R$ and $\Lambda$ respectively, so that
\beq\label{eq:power.rings}R\llbracket t \rrbracket=\Z\llbracket t,M^{(1)},M^{(2)}, \cdots ~\rrbracket\and \Lambda\llbracket t \rrbracket=\Z\llbracket t , \sfh_1,\sfh_2,\cdots~\rrbracket.\eeq
The above maps $\{\ch_{\symS_n}\}_{n\geq 0}$ naturally extend to the \emph{Frobenius characteristic} map 
$\ch :R\llbracket t\rrbracket \lra \Lambda\llbracket t\rrbracket$, which is a graded ring homomorphism.

\subsection{Plethysm}\label{ss:plethysm}
Plethysm is another associative product 
on $\Lambda$
denoted by $\circ$, which is uniquely determined by the following properties:
\begin{enumerate}
	\item $F\circ \sfp_n=\sfp_n\circ F=F(x_1^n,x_2^n,\cdots)$
	\item $(F+G)\circ H=F\circ H+G\circ H$
	\item $(FG)\circ H=(F\circ H)(G\circ H)$
\end{enumerate}
for $F,G,H\in \Lambda$ and $n \geq1$. Note that $\sfh_1=\sfp_1$ is the two-sided identity by (1).
This naturally extends to a multiplication on $\Lambda\llbracket t\rrbracket $ satisfying
\begin{enumerate}
	\item[(1$'$)] $F\circ \sfp_n=\sfp_n\circ F = F(t^n,x_1^n,x_2^n,\cdots)$ and $t\circ F=t$
\end{enumerate}
and the properties (2)--(3)
for $F
, G, H\in \Lambda\llbracket t\rrbracket $ and $n\geq 1$.
For more details, see \cite[\S I.8]{Mac}, \cite[\S5.2]{Get} or \cite[\S7.2]{GeKa}.

\medskip
We will mostly use the operation $\sfh_r\circ(-)$ with $r\geq0$. Note that when $r=0$, we have $\sfh_0\circ F=1$ for every $F\in \Lambda\llbracket t\rrbracket $ by the property (3).

Since $\sfh_r = \sum_{\lambda\vdash r} z_\lambda^{-1} \sfp_\lambda$, by its properties the plethysm can be explicitly computed as
\[\sfh_r\circ F = \sum_{\lambda\vdash r} z_\lambda^{-1} \sfp_\lambda \circ F = \sum_{\lambda\vdash r} z_\lambda^{-1} \prod_{i=1}^{\ell(\lambda)} F(t^{\lambda_i},z_1^{\lambda_i},z_2^{\lambda_i},\cdots). \]
In particular, one can check that $\sfh_r\circ (tF) =t^r(\sfh_r\circ F)$.
\medskip

For $F=\sum_{n,k\geq 0}F_{n,k}\in \Lambda\llbracket t\rrbracket $ with $F_{n,k}\in \Lambda_n$, write
\[F_{\leq (n,k)}:=\sum_{0\leq i\leq n,\; 0\leq j\leq k}F_{i,j}t^j\]
for its truncation up to degrees $(n,k)$.
\begin{lemma}\label{lem:plethysm.exp}
	Let $r\geq 1$. Let $(F_j)_{j\in \Z_{\geq 1}}$ be a sequence of elements $F_j\in \Lambda\llbracket t\rrbracket $ such that $F:=\sum_{j\geq 1}F_j$ is well defined, that is, for every pair $n,k$, there exist only finitely many $F_j$ with nonzero components in  $\Lambda_nt^k$. Then,
    \beq \label{eq:plethysm.exp} \sfh_r\circ F=\sum_{\substack{(r_1,\cdots,r_m)\vdash r}}\;\sum_{\substack{i_1,\cdots,i_m>0\\ \text{distinct}}}\;\prod_{j=1}^m \sfh_{r_j}\circ F_{i_j}\eeq
 where $(r_1,\cdots,r_m)$ runs over all the partitions of $r$ (with $m,r_j>0$).
\end{lemma}
\begin{proof}
	The assertion is known to be true when $F=\sum_{j=1}^mF_j$ is a finite sum:
	one can check this holds by using the formula (\cite[(I.8.8)]{Mac})
	\beq\label{eq:plethysm.exp.two} \sfh_r\circ(F+G)=\sum_{i=0}^r(\sfh_{r-i}\circ F)(\sfh_{i}\circ G)\eeq
	for $F,G\in \Lambda\llbracket t\rrbracket $ and by the induction on the number $m$ of summands.	
	
	On the other hand, for every $n,k>0$, we have an equality
	\beq\label{eq:trunc}\left(\sfh_r\circ F\right)_{\leq (n,k)}=\left(\sfh_r\circ F_{\leq (n,k)}\right)_{\leq (n,k)}\eeq
	for any $F\in \Lambda\llbracket t\rrbracket $.
	Indeed, if we let $G=F-F_{\leq (n,k)}$, then by \eqref{eq:plethysm.exp.two},
	\[\sfh_r\circ F=\sfh_r\circ F_{\leq (n,k)}+\sum_{i=1}^r(\sfh_{r-i}\circ F_{\leq (n,k)})(\sfh_{i}\circ G)\]
	with $(\sfh_i\circ G)_{\leq (n,k)}=0$ for $i>0$.
	By \eqref{eq:trunc}, the assertion holds if it does under the finiteness assumption on the summands. This completes the proof.
\end{proof}
\begin{example}[$r=2$]\label{ex:antisym} Let $F\in \Lambda\llbracket t\rrbracket $ be as above. Then,
	\[\sfh_2\circ \sum_{i}F_i=\sum_i \sfh_2\circ F_i +\sum_{i<j}F_iF_j.\]
	Recall that $\sfs_{(1,1)}$ is the Frobenius characteristic  of the (one-dimensional) sign representation of $\symS_2$ and satisfies
	$\sfs_{(1,1)}=\sfh_1^2-\sfh_2$.
	Hence,
	we have
	\beq \label{eq:antisymm}\sfs_{(1,1)}\circ F=F^2-\sfh_2\circ F=\sum_{i}\sfs_{(1,1)}\circ F_i+\sum_{i<j}F_iF_j.\eeq
\end{example}

The following shows how the plethysm $\sfh_r\circ(-)$ reads on $R$ through $\ch$.

\begin{lemma}\label{lem:plethysm.Ind}
	Let $r,n\geq1$.
	Let $V\in R_n$ be a representation of $\symS_n$. 
	Then,
	\[\sfh_r\circ \ch_{\symS_n}(V)=\ch_{\symS_{rn}}\left(\Ind^{\symS_{rn}}_{(\symS_n)^r\rtimes \symS_r}V^{\otimes r}\right)\]
	where the $(\symS_n)^r\rtimes \symS_r$-action on $V^{\otimes r}$ is defined by the component-wise action of $(\symS_n)^r$ and the action of $\symS_r$ permuting the factors of $V^{\otimes r}$.
\end{lemma}
\begin{proof}
	See \cite[Remark I.8.2 and (I.A.6.2)]{Mac}.
\end{proof}

\subsection{Plethystic exponential}\label{ss:Exp}
We recall the plethystic exponential, denoted by $\Exp(-)$, and its property. For more details, see \cite[\S8.4]{GeKa}.

Denote by $\Lambda\llbracket t\rrbracket _+\subset \Lambda\llbracket t\rrbracket $ the subgroup consisting of elements with no constant terms.
\begin{definition}[Plethystic exponential]
	For $F\in \Lambda\llbracket t\rrbracket _+$,
	define
	\[\Exp(F):=\exp\sum_{r\geq 1}\frac{\sfp_r}{r}\circ F =\sum_{r\geq0}\sfh_r\circ F.\]
\end{definition}
\noindent This is an element of $1+\Lambda\llbracket t\rrbracket _+\subset \Lambda\llbracket t\rrbracket $, since $\sfh_0\circ(-)=1$.

\begin{lemma}\label{lem:Exp.homo}
	$\Exp(F+G)=\Exp(F)\Exp(G)$ for $F$, $G\in \Lambda\llbracket t\rrbracket _+$. In particular, $\Exp(-F)$ is the multiplicative inverse of $\Exp(F)$.
\end{lemma}
\begin{proof}
	By Lemma~\ref{lem:plethysm.exp} or \eqref{eq:plethysm.exp.two}, we have
	\[\begin{split}
		\Exp(F+G)&=\sum_{r\geq0}\sum_{i+j=r,\;i,j\geq0}(\sfh_i\circ F)(\sfh_j\circ G)
		=\Exp(F)\Exp(G).
	\end{split}
	\]
	The second assertion holds as $\Exp(F)\Exp(-F)=\Exp(0)=1$.
\end{proof}
\begin{example} \label{ex:plethysm.-h1}
	$\Exp(\sfh_1 )=\sum_{r\geq0}\sfh_r$, since $\sfh_r\circ \sfh_1=\sfh_r$ for every $r\geq 0$.
	Let
	\[\sfe_r:=\sum_{1\leq i_1<\cdots<i_r}x_{i_1}\cdots x_{i_r}\]
	denote the $r$-th elementary symmetric function, where $\sfe_0=1=\sfh_0$. Then,
 	$\sfe_r= \sfs_{(1^r)}$ and $0=\sum_{i=0}^r(-1)^i\sfe_i\sfh_{r-i}$ for $r>0$. (\cite[I.(2.6), I.(3.9)]{Mac}) In particular, we have
	\[\Exp(-\sfh_1)=\sum_{r\geq0}(-1)^r\sfe_r\]
	since it is the multiplicative inverse of $\Exp(\sfh_1)$. By the same argument,
	\[\sfh_r\circ (-F)=(-1)^r \sfe_r\circ F\]
	for every $F\in \Lambda\llbracket t\rrbracket $.
	See also \cite[Examples~I.8.1.(a)]{Mac}.
\end{example}

\begin{proposition}\cite[Proposition~8.5]{GeKa}
	The map $\Exp:\Lambda\llbracket t\rrbracket _+\to 1+\Lambda\llbracket t\rrbracket _+$ is invertible over $\Q$, with the inverse
	\beq \label{eq:Log}\Log(F):=\sum_{r\geq1}\frac{\mu(r)}{r}\log (\sfp_r)\circ F=\sum_{r\geq1}\frac{\mu(r)}{r}\log \left(\sfp_r\circ F\right)\eeq
	where $\mu$ is the M\"obious function and $\log$ is the usual log function.
\end{proposition}

\bigskip

\section{Representations on the cohomology of $\Mbar_{0,n}$}
\label{s:review}
In this section, we review necessary results from \cite{CKL}. Let $n\ge 3$ and consider the forgetful morphism
\[\pi:\Mbar_{0,n+1}\lra\Mbar_{0,n}\]
which forgets the last marking. The symmetric group $\symS_n$ acts on $\Mbar_{0,n}$ (resp.  $\Mbar_{0,n+1}$) by permuting the $n$ markings (resp. the first $n$ markings fixing the last), and the map $\pi$ is equivariant under these actions.

In \cite{CKL}, we established an $\symS_n$-equivariant factorization of $\pi$ using the theory of $\delta$-stable quasimaps which gives us a formula comparing the $\symS_n$-modules $H^*(\Mbar_{0,n})$ and $H^*(\Mbar_{0,n+1})$. Then we developed a combinatorial algorithm to compute the $\symS_n$-representations on the cohomology of $\Mbar_{0,n+1}$ by employing the Kapranov map \[ \Mbar_{0,n+1} \lra \PP^{n-2},\]which can be factored into a sequence of $\symS_n$-equivariant blowups. Combining these two results provides an algorithm for computing the $\symS_n$-equivariant cohomology of $\Mbar_{0,n}$.

\subsection{Moduli of $\d$-stable quasimaps}\label{ss:quasimap}
The moduli space of $\delta$-stable quasimaps was defined and studied in \cite{CK}. A \emph{quasimap} is a triple $(C, L, s)$ consisting of
\begin{itemize}
	\item nodal curves $C$ of genus $g$ with $n$ distinct marked points,
	\item line bundles $L$ on $C$ of degree $d$,
	\item nonzero multisections $s\in H^0(C,L^{\oplus m})$ of $L$.
\end{itemize}
The \emph{$\d$-stability} for a quasimap was introduced in \cite{CK} for any positive rational number $\d$, and then the moduli space $\fQ^\d$ of $\d$-stable quasimaps and its virtual invariant were constructed. For precise definitions and details, we refer to \cite{CK}.

In \cite{CK, CKL}, we studied the wall crossings of the moduli spaces of $\d$-stable quasimaps in the special case where $g=0$ and $m=d=1$, which give us an $\symS_n$-equivariant factorization of the forgetful morphism $\pi:\Mbar_{0,n+1}\to\Mbar_{0,n}$.

\begin{theorem}\cite[Theorem~2.8]{CKL}
\label{thm:qwc.geom}
	Let $n\geq 3$ and $\ell = \lfloor \frac{n-1}{2}\rfloor$. The forgetful morphism $\pi$ factorizes as
	\beq \label{eq:fact.pi}\Mbar_{0,n+1}\cong \fQ^{\d=\infty}\xrightarrow{~\rho_2~} \cdots \xrightarrow{~\rho_\ell~} \fQ^{\d=0^+}\xrightarrow{~p~} \Mbar_{0,n}\eeq
where
	\begin{enumerate}
		\item for $2\leq i\leq \ell$, $\rho_i$ is the blowup along the disjoint union of $\binom{n}{i}$ copies of $\Mbar_{0,i+1}\times \Mbar_{0,n-i+1}$,
		\item when $n=2\ell+1$, $p$ is a $\PP^1$-bundle morphism;\\
		when $n=2\ell+2$, $p$ is the blowup of a $\PP^1$-bundle over $\Mbar_{0,n}$ along the disjoint union of $\frac{1}{2}\binom{n}{\ell+1}$ copies of $\Mbar_{0,\ell+2}\times \Mbar_{0,\ell+2}$.
	\end{enumerate}
	All $\rho_i$ and $p$ are $\symS_n$-equivariant.
\end{theorem}

As $\d$ varies, the moduli space changes only at finitely many values of $\d$, which are called the \emph{walls}. Since the degree of the line bundle $L$ is one, the section $s$ is determined by the unique point on $C$ where it vanishes. When $\d$ is sufficiently large (denoted by $\d=\infty$), one can deduce that $\fQ^{\d=\infty}\cong \Mbar_{0,n+1}$ (cf. \cite[Lemma~2.7]{CKL}), where the vanishing point of the section $s$ is treated as an additional marking.

As $\d$ decreases, there are $\ell -1$ walls where the moduli space varies. If a quasimap becomes unstable as we cross a wall, there is a \emph{modification} procedure to obtain a stable quasimap. (cf. \cite[\S 7.2]{CK}) When $d=1$, such procedure induces a blowup morphism at each wall along the locus of unstable quasimaps \cite[Proposition 7.10]{CK}.

When $\d$ is sufficiently close to zero (denoted by $\d=0^+$) and $n$ is odd, a quasimap $(C,L,s)$ is $\d$-stable if $C\in \Mbar_{0,n}$ and $L$ has degree one on the unique central component of $C$. Here a central component means that its complement has no connected subcurve with more than $\frac{n}{2}$ markings. When $n$ is even, the curve $C$ may have a central node instead of a central component and we need to blow up along this locus to get $\fQ^{\delta=0^+}$.

Let
\[P_n=\sum_{k=0}^{n-3}\ch_{\symS_n} \left(H^{2k}\left(\Mbar_{0,n}\right)\right)t^k \and Q_n=\sum_{k=0}^{n-2}\ch_{\symS_n} \left(H^{2k}\left(\Mbar_{0,n+1}\right)\right)t^k\]
be the Frobenius characteristics of the graded $\symS_n$-representations and let
\[P=1+\sfh_1+\sfh_2+\sum_{n\geq 3}P_n \and Q=1+\sfh_1+\sum_{n\geq2}Q_n.\]

By Theorem~\ref{thm:qwc.geom}, 
we obtain (cf. \cite[Proposition 4.1]{CKL})
\beq\label{eq:qwc1}Q_n=P_{\fQ^{\d=0^+}} + t\sum_{2\leq i<\frac{n}{2}}Q_iQ_{n-i},  \eeq
where \[P_{\fQ^{\d=0^+}}(t):=\sum_{k=0}^{n-2}\ch_{\symS_n}\left(H^{2k}\left(\fQ^{\d=0^+}\right)\right)t^k\]
is defined similarly using the $\symS_n$-action on $\fQ^{\d=0^+}$.

Moreover, it is immediate that
$P_{\fQ^{\d=0^+}}=(1+t)P_n$ when $n$ is odd, by the projective bundle formula.
When $n$ is even, we can use an $\symS_n$-equivariant description of $\fQ^{\d=0^+}$ as a GIT quotient to prove (cf. \cite[Corollary~4.7]{CKL})
\beq\label{eq:qwc2}P_{\fQ^{\d=0^+}}=(1+t)P_n+t\sfs_{(1,1)}\circ Q_{\frac{n}{2}}\eeq
for any $n$, where we set $Q_{\frac{n}{2}}=0$ for $n$ odd. 
For details, we refer to \cite[\S4.3--4.4]{CKL}.

Combining \eqref{eq:qwc1} and \eqref{eq:qwc2}, we get the formula which relates $P_n$ and $Q_n$.

\begin{theorem}\cite[Theorem~4.8]{CKL}\label{thm:CKL} For $n\geq 3$,  we have 
	\beq \label{eq:CKL}(1+t)P_n=Q_n -t\left(\sum_{2\leq i<\frac{n}{2}}Q_iQ_{n-i}+\sfs_{(1,1)}\circ Q_{\frac{n}{2}} \right)\eeq
	where we set $Q_{\frac{n}{2}}=0$ for $n$ odd.
\end{theorem}

From this, we deduce a formula relating $P$ and $Q$.
\begin{corollary}\label{cor:CKL.gen}  $P$ and $Q$ satisfy
\beq \label{eq:qwc.gen}
(1+t)P=(1+t+\sfh_1 t)Q-t\sfs_{(1,1)}\circ Q.
\eeq
\end{corollary}
\begin{proof}
We first show that, if we let
	\[\bar P:=\sum_{n\geq 3}P_n \and
	\bar Q:=\sum_{n\geq 2}Q_n,\] then we have
	\beq \label{eq:qwc.gen2} (1+t)\bar P= (\bar Q-\sfh_2)-t\sfs_{(1,1)}\circ \bar Q.\eeq
Indeed, by \eqref{eq:antisymm} and Theorem \ref{thm:CKL}, we have
	\[\begin{split}
		t\sfs_{(1,1)}\circ \bar Q&=t\left(\sum_{2\leq n_1<n_2}Q_{n_1}Q_{n_2}+\sum_{n\geq2}\sfs_{(1,1)}\circ Q_n\right)\\
		&=\sum_{n\geq4}t \left(\sum_{2\leq i<\frac{n}{2}}Q_iQ_{n-i}+\sfs_{(1,1)}\circ Q_{\frac{n}{2}}\right)\\
		&=\sum_{n\geq4}\Big(Q_n-(1+t)P_n\Big)= (\bar Q -\sfh_2) -(1+t)\bar P,
	\end{split}\]
	where the last equality holds because $Q_3=(1+t)\sfh_3=(1+t)P_3$ and $Q_2=\sfh_2$.
Now by substituting $\bar P=P-(1+\sfh_1 +\sfh_2 )$ and $\bar Q=Q-(1+\sfh_1 )$ in \eqref{eq:qwc.gen2}, \eqref{eq:qwc.gen} follows.
\end{proof}

As a result, the computation of $P$ is reduced to that of $Q$.

\subsection{$\symS_n$-representations on the cohomology of $\Mbar_{0,n+1}$}
As first observed by Kapranov \cite{Kap}, $\Mbar_{0,n+1}$ admits a birational morphism
onto $\PP^{n-2}$, which factorizes into a sequence of blowups
\beq\label{eq:Kap.map}\Mbar_{0,n+1}\xrightarrow{f_{n-2}}\cdots \xrightarrow{~f_1~}\PP^{n-2},\eeq
where the intermediate spaces are Hassett's moduli spaces of  weighted stable curves, with weights of the form $(({1}/{k})^n,1)$ for $2\le k\le n-1$.

Geometrically, these blowup morphisms are described as follows. For given $n$ points in $\PP^{n-2}$ in general position, $f_i$ is the blowup along the transversal union of the proper transforms of $\binom{n}{i}$ distinct $(i-1)$-planes $\PP^{i-1}$ passing through $i$ points among the $n$ points for each $i$. For example, $f_1$ is the blowup along the $n$ points, $f_2$ is the blowup along the proper transforms of the $\binom{n}{2}$ lines passing through pairs of the $n$ points, etc.

For such $n$ points, the $\symS_n$-action permuting the $n$ points naturally extends to a unique projective linear automorphism of $\PP^{n-2}$. With respect to this action, \eqref{eq:Kap.map} is $\symS_n$-equivariant.

As proved by Hassett in \cite{Has}, \eqref{eq:Kap.map} enjoys a further nice property: irreducible components of the blowup centers of $f_i$, and their nonempty (transversal) intersections also admit birational morphism to the projective spaces, together with factorization into blowups along the transversal unions of smooth subvarieties. Moreover, these subvarieties again admit birational morphisms to the projective spaces together with factorizations by such nice blowups.

Consequently, by applying the blowup formula to all these blowups, one can express the cohomology of $\Mbar_{0,n+1}$ in terms of the cohomology of various projective spaces that appear as the targets in the above description. This process is systematically carried out in \cite[\S5]{CKL} using the blowup formula for the transversal union of smooth subvarieties \cite[Proposition~6.1]{BM}. The combinatorics of the blowups is governed by rooted trees, which we review below, leading to the result that
\beq \label{eq:cohom.M0n+1.2}H^*(\Mbar_{0,n+1})\cong \bigoplus_{T\in \sT^o_n}H^*\left(\PP^{val(v_0)-3}\right)\otimes \bigotimes_{v\in V(T)\setminus\{v_0\}}H^+\left(\PP^{val(v)-3}\right)\eeq
where $\sT_n^o$ denotes the set of rooted trees with $n$ legs, $v_0$ denotes the root of $T\in \sT_n^o$, and $H^+(\PP^m)\subset H^*(\PP^m)$ denotes the subspace of positive cohomology degrees (cf. \cite[(5.9)]{CKL}). By taking the degree $2k$ part, we obtain the following.
\begin{proposition}\cite[Proposition~5.12]{CKL}\label{prop:Q.wrt} The $\symS_n$-representation on $H^{2k}(\Mbar_{0,n+1})$ is the permutation representation of $\symS_n$ spanned by weighted rooted trees with $n$ inputs and weight $k$.
\end{proposition}

We now provide a review of weighted rooted trees. 
A \emph{tree} is a connected graph with no loops or cycles. We allow some edges of a tree to be attached to only one end vertex. We refer to such edges as \emph{legs}, while reserving the term ``edges'' for those with two end vertices. 
For a vertex $v$ of a tree, the \emph{valency} of $v$, denoted by $val(v)$, is the number of edges and legs attached to $v$.

\begin{definition}\label{def:rooted.tree}
	A \emph{labeled rooted tree} or \emph{$n$-labeled rooted tree} is a tree with one distinguished leg called the \emph{output}, and with the remaining legs, called the \emph{inputs}, labeled by integers $1,\dots,n$, where $n$ is the number of inputs.
	The vertex to which the output is attached is called the \emph{root}.
	A \emph{rooted tree} is a tree obtained from a labeled rooted tree by forgetting the labels (and order) on the inputs.
\end{definition}

Our notion of a rooted tree (resp. $n$-labeled rooted tree) is the same as that of a tree (resp. $n$-tree) in \cite[\S1.1.1]{GiKa}.
\begin{figure}[h]
    \centering
    \begin{tikzpicture} [scale=0.8,auto=left,every node/.style={scale=0.8}]
      \tikzset{Bullet/.style={circle,draw,fill=black,scale=0.5}}
      \node[Bullet] (n0) at (2,0.5) {};
      \node[Bullet] (n1) at (2,-0.5) {};
      \node[] (n2) at (0.5,-1.5) {};
      \node[] (n3) at (1.5,-1.5) {};
      \node[] (n4) at (2.5,-1.5) {};
      \node[] (n5) at (3.5,-1.5) {};

      \draw[black] (n0) -- (2,1.5);
      \draw[black] (n0) -- (n1);
      \draw[black] (n1) -- (n2);
      \draw[black] (n1) -- (n3);
      \draw[black] (n1) -- (n4);
      \draw[black] (n1) -- (n5);
      \draw[black] (n0) -- (0.5,-0.5);
      \draw[black] (n0) -- (3.5,-0.5);
      \draw[] (2,1.5) node[right] {the output};
      \draw (2.1,0.5) node[right] {$v_0$};
      \draw (2.1,-0.5) node[right] {$v$};
      \draw[] (0.4,-0.5) node {$1$};
      \draw[] (3.6,-0.5) node {$2$};
      \draw[] (0.4,-1.6) node {$3$};
      \draw[] (1.4,-1.6) node {$4$};
      \draw[] (2.6,-1.6) node {$5$};
      \draw[] (3.6,-1.6) node {$6$};
    \end{tikzpicture}
    \qquad
    \begin{tikzpicture} [scale=0.8,auto=left,every node/.style={scale=0.8}]
      \tikzset{Bullet/.style={circle,draw,fill=black,scale=0.5}}
      \node[Bullet] (n0) at (2,0.5) {};
      \node[Bullet] (n1) at (2,-0.5) {};
      \node[] (n2) at (0.5,-1.5) {};
      \node[] (n3) at (1.5,-1.5) {};
      \node[] (n4) at (2.5,-1.5) {};
      \node[] (n5) at (3.5,-1.5) {};

      \draw[black] (n0) -- (2,1.5);
      \draw[black] (n0) -- (n1);
      \draw[black] (n1) -- (n2);
      \draw[black] (n1) -- (n3);
      \draw[black] (n1) -- (n4);
      \draw[black] (n1) -- (n5);
      \draw[black] (n0) -- (0.5,-0.5);
      \draw[black] (n0) -- (3.5,-0.5);
      \draw[] (2,1.5) node[right] {the output};
      \draw (2.1,0.5) node[right] {$v_0$};
      \draw (2.1,-0.5) node[right] {$v$};
      \draw[] (0.4,-0.5);
      \draw[] (3.6,-0.5);
      \draw[] (0.4,-1.6);
      \draw[] (1.4,-1.6);
      \draw[] (2.6,-1.6);
      \draw[] (3.6,-1.6);
    \end{tikzpicture}
    \caption{A labeled rooted tree and a rooted tree}
    \label{fig:RootedTree}
\end{figure}
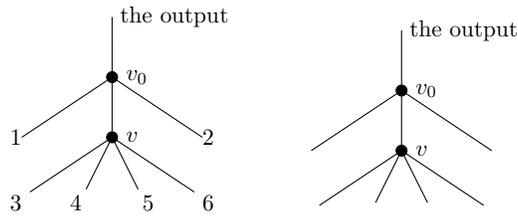
The labeled rooted tree and the rooted tree in Figure \ref{fig:RootedTree} are examples with six inputs. Each has two vertices $v_0$ and $v$ of valencies four and five respectively, where $v_0$ is the root.

\begin{definition}\label{def:wrt} Let $T$ be a labeled rooted tree or a rooted tree in Definition~\ref{def:rooted.tree}.
	A \emph{weight function} on $T$ is a $\Z_{\geq 0}$-valued function
	\[w:V(T)\lra \Z_{\geq 0}\]
	such that for each $v\in V(T)$,
	\begin{enumerate}
		\item $0\leq w(v)\leq val(v)-3$ if $v$ is the root;
		\item $0< w(v) \leq val(v)-3$ if $v$ is not the root.
	\end{enumerate}
	
	We call such a pair $(T,w)$ a (resp. \emph{labeled}) \emph{weighted rooted tree} if $T$ is a (resp. labeled) rooted tree. The \emph{weight} of $(T,w)$ refers to $\sum_{v\in V(T)}w(v)$.
\end{definition}
Note that only the root can have zero weight, whereas weights on all the other vertices are always positive. The conditions~(1) and (2) on the weight function are chosen based on the formula \eqref{eq:cohom.M0n+1.2}.

\begin{definition}\label{def:set.wrt}
	For $n,k\geq0$, we denote by $\sT_{n,k}$ (resp.~$\sT_{n,k}^{\mathrm{lab}}$) the set of (resp. labeled) weighted rooted trees with $n$ inputs and weight $k$.
\end{definition}

There is a natural action of $\symS_n$ on $\sT_{n,k}^{\mathrm{lab}}$ permuting the labels on the inputs. The quotient map by this action
\[\Lab:\sT_{n,k}^{\mathrm{lab}}\lra \sT_{n,k}^{\mathrm{lab}}/\symS_n= \sT_{n,k}\]
is the forgetful morphism which forgets the labels on the inputs.

\begin{remark}
	We slightly changed the notations from \cite{CKL} for simplification. For example, $\sT_{n,k}^{\mathrm{lab}}$ defined above is the set denoted by $\sT_{n,k}$ in \cite{CKL}.
	\end{remark}

\begin{example} \label{ex:empty.cases}
	If $n\leq 1$, then $\sT_{n,k}=\emptyset$ for any $k$. If $n\geq 2$, then $\sT_{n,0}$ consists of the unique weighted rooted tree which has only one vertex, the root. 
\end{example}

Note that every fiber of $\Lab$ is the set of all possible ways of labeling the inputs of a given rooted tree by $1,\cdots,n$, and hence it is
 an $\symS_n$-set.
\begin{definition}\label{def:U_T}
	For $(T,w)\in \sT_{n,k}$, 
	we define  $U_{(T,w)}$ to be the permutation representation of $\symS_n$ spanned by 
	$\Lab^{-1}{(T,w)}$.
	An equivalent definition
	is that
	\[U_{(T,w)}=U_{\Stab(T,w)}=\Ind^{\symS_n}_{\Stab(T,w)}\mathbbm{1}\]
	as in Definition~\ref{def:perm.rep}, where $\Stab(T,w)$ denotes the stabilizer subgroup in $\symS_n$ of an element in $\Lab^{-1}{(T,w)}$ so that $\Lab^{-1}(T,w)\cong \symS_n/\Stab(T,w)$.
	
\end{definition}

For convenience, we will sometimes omit $w$ in our notations, and write $T$, $\Stab(T)$ and $U_T$ for $(T,w)$, $\Stab(T,w)$ and $U_{(T,w)}$ respectively.

\begin{example}\label{ex:UT1}
	Suppose that  $T\in \sT_{n,k}$ has only one vertex, the root. Then, $\Lab^{-1}(T)$ consists of a single element, and $\Stab(T)=\symS_n$. Hence, $U_T\cong M^{(n)}$ and
	$\ch_{\symS_n}(U_T)=\sfh_n$.
\end{example}
\begin{example}
	Suppose that  $T\in \sT_{n,k}$ has two vertices with $a$ inputs attached to the root. Then $\Lab^{-1}(T)$ consists of $\binom{n}{a}=\frac{n!}{a!(n-a)!}$ elements, and $\Stab(T)\cong \symS_a\times \symS_{n-a}$. Hence, $U_T\cong M^{(a,n-a)}$ and
	$\ch_{\symS_n}(U_T)=\sfh_{(a,n-a)}$.
\end{example}

\begin{example}
	Suppose that $(T,w)\in \sT_{n,k}$ has three vertices, and that the two non-root vertices are adjacent to the root and have the same weights and the same number $a\leq \frac{n}{2}$  of inputs attached to them. Then, $\Lab^{-1}{(T,w)}$ consists of $\frac{1}{2}\binom{n}{a}\binom{n-a}{a}=\frac{1}{2}\frac{n!}{(a!)(a!)(n-2a)!}$ elements, and $\Stab(T,w)\cong \symS_{n-2a}\times ((\symS_a\times \symS_a)\rtimes \symS_2)$. Hence, $\ch_{\symS_n}(U_T)=\sfh_{n-2a}\cdot (\sfh_2\circ \sfh_a)$.
\end{example}

By Proposition \ref{prop:Q.wrt}, we have an isomorphism
	\[H^{2k}(\Mbar_{0,n+1})\cong \bigoplus_{(T,w)\in \sT_{n,k}}U_{(T,w)}\]
	 of $\symS_n$-representations. Hence,
\beq\label{eq:Qnk}Q_{n,k}=\sum_{(T,w)\in \sT_{n,k}}\ch_{\symS_n}(U_{(T,w)}).\eeq

Our primary goal now is to compute the generating function $Q$ using the combinatorial formula described above. This will lead us to the computation of $P$ by the wall crossing formula \eqref{eq:qwc.gen}. In the next section, we develop a recursive algorithm for $Q$ based on the recursive structure of the weighted rooted trees.

\bigskip

\section{Recursive algorithm} 
\label{s:recursion}
In this section, we capture an inherent recursive structure on $Q$, based on the combinatorial formula \eqref{eq:Qnk}. 

\subsection{Recursive structure}\label{ss:key.lemmas} We present key lemmas for our recursion.

\begin{definition}\label{def:Tplus}
	Let $\sT_{n,k}^+\subset \sT_{n,k}$ denote the subset consisting of weighted rooted trees with positive weight at the root.
	In other words, $\sT_{n,k}^+$ is the set of pairs $(T,w)\in \sT_{n,k}$ satisfying
	\beq\label{eq:wt.cond.pos}
		0< w(v)\leq val(v)-3
	\eeq
	for every vertex $v$ of $T$.
\end{definition}

\begin{example}\label{ex:empty.cases+}
	$\sT^+_{n,k}=\emptyset$ if either $n\leq 2$ or $k=0$. 
	If $n\geq 3$, then $\sT^+_{n,1}$ consists of the unique weighted rooted tree with only one vertex, the root.
\end{example}

\begin{lemma} \label{lem:key1}
Let
\[\sT:=\bigsqcup_{n,k\geq0}\sT_{n,k} \and \sT^+:=\bigsqcup_{n,k\geq0}\sT^+_{n,k}.\] Then there is a natural bijection
\[\Phi:\prod_{\substack{a,r\geq0\\0\le b\le a+r-2}}(\sT^+)^{\times r}/\symS_r\xrightarrow{~\cong~}\sT\]
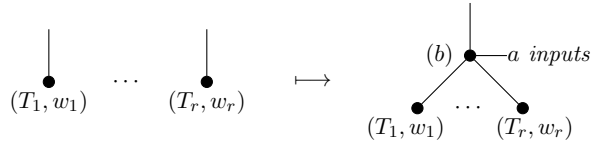
\[\hspace{2em}
\begin{tikzpicture} [scale=.7,auto=left,every node/.style={scale=0.8}]
      \tikzset{Bullet/.style={circle,draw,fill=black,scale=0.5}}
      \node[Bullet] (T1) at (-7,-1) {};
      \node[Bullet] (T2) at (-4,-1) {};
      \node[Bullet] (n0) at (1,-0.5) {};
      \node[Bullet] (n1) at (0,-1.5) {};
      \node[Bullet] (n2) at (2,-1.5) {};

      \draw[black] (-5.5,-1) node[] {$\cdots$};
      \draw[black] (T1) -- (-7,0);
      \draw[black] (T2) -- (-4,0);
      \draw[black] (T1) node[below] {$(T_1,w_1)$};
      \draw[black] (T2) node[below] {$(T_r,w_r)$};
      \draw[black] (-2.5,-1) node[right] {$\longmapsto$};
      \draw[black] (n0) -- (1,0.5);
      \draw[black] (n0) -- (n1);
      \draw[black] (n0) -- (n2);
      \draw[black] (1,-1.5) node[] {$\cdots$};
      \draw[black] (n1) node[below] {$(T_1,w_1)\quad$};
      \draw[black] (n0) -- (1.7,-0.5) node[near end, right] {$a$ inputs};
      \draw[black] (n0) node[left] {$(b)~$};
      \draw[black] (n2) node[below] {$\quad(T_r,w_r)$};
    \end{tikzpicture}
    \]
	 which sends the $\symS_r$-equivalence class of the $r$-tuple $((T_1,w_1),\cdots ,(T_r,w_r))\in \sT^+_{n_1,k_1}\times \cdots\times \sT^+_{n_r,k_r}$ of weighted rooted trees, together with a pair $(a,b)$ of integers, to a new weighted rooted tree $(T,w)$
	 where
	 \begin{enumerate}
	 	\item $T$ is obtained by attaching the outputs of $T_i$ to the root $v$ of the rooted tree with only one vertex $v$ and $a$ inputs attached to $v$;
	 	\item $w$ is defined by $w(v):=b$ and $w(v'):=w_i(v')$ if $v'\in V(T_i)$.
	 \end{enumerate}
	 In particular,
	 \[val(v)=a+r+1, \quad V(T)=\{v\}\cup \bigsqcup_{i=1}^rV(T_i) \and  (T,w)\in \sT_{n+a,k+b},\] where $n=n_1+\cdots+n_r$ and $k=k_1+\cdots +k_r$. When $r=0$, we assume that $(\sT^+)^{\times r}/\symS_r$ consists of one element, and $\Phi$ sends it to a weighted rooted tree with one vertex.

By restricting $\Phi$ to the subset where $b>0$, we get a natural bijection
	 \[\Phi^+:\prod_{\substack{a,r\geq0,\\0<b\le a+r-2}}(\sT^+)^{\times r}/\symS_r\xrightarrow{~\cong~}\sT^+.\]
\end{lemma}

\begin{proof}
	We prove that $\Phi$ is bijective. The same proof will work for $\Phi^+$.
	
	We first show that $\Phi^+$ is well defined. Note that any two elements in the same $\symS_r$-equivalence class have the same image, since inputs are unordered in our definition of weighted rooted trees. Since $0\le w(v)=b\leq a+r-2=val(v)-3$, the valency condition holds for the root $v$ of $(T,w)$ and by definition it is already satisfied for the other non-root vertices. Hence the  image $(T,w)$ is in $\sT$.

	The inverse map is obvious. Given $(T,w)\in \sT$ with the root $v$, we let $a$ be the number of inputs attached to $v$ and $b=w(v)\ge 0$. We associate the $\symS_r$-equivalences classes of the set of $r:=val(v)-1-a$ weighted rooted trees obtained by removing $v$ and $a$ inputs attached to $v$. One can check that this is indeed the inverse.
\end{proof}

For simplicity, we omit $w$ from the notation $(T, w)$ below and use $T$ to refer to weighted rooted trees.
\begin{lemma}\label{lem:key2}
	Let $a$, $b$ and $r$ be nonnegative integers with $0 \le b \le a+r-2$. Let $(r_1,\cdots,r_m)$ be a partition of $r$. For distinct $T_1,\cdots, T_m\in \sT^+$, let $T_\bullet \in (\sT^+)^r/\symS_r$ denote the element consisting of $r_j$ copies of $T_j$. Then,
 \[\ch\left(U_{\Phi(T_\bullet)}\right)=\sfh_a\cdot \prod_{j=1}^m\sfh_{r_j}\circ \ch_{n_j}\left(U_{T_j}\right)\]
 where $n_j$ are the numbers of inputs of $T_j$. We set $\prod_{j=1}^m(-)=1$ for $m=0$.
\end{lemma}

\begin{proof}
	This is immediate from Definition~\ref{def:U_T} and Lemma~\ref{lem:plethysm.Ind}, since we have \[\Stab(\Phi(T_\bullet))\cong \symS_a\times \prod_{j=1}^m \left(\Stab(T_{i_j})^{r_j}\rtimes \symS_{r_j}\right)\] by the construction of $\Phi$.
\end{proof}

\subsection{Recursion for $Q$}
We establish a recursion for $Q$ in this subsection. 

We begin with definitions of $Q_{n,k}^+$, $Q_n^+$ and $Q^+$.
\begin{definition}
	For $n\ge 2$, let
	\[Q_{n,k}^+:=\sum_{(T,w)\in \sT_{n,k}^+}\ch_{\symS_n}(U_{(T,w)}) \and Q^+_n:=\sum_{k=1}^{n-2}Q_{n,k}^+t^k.\]
	Define \[Q^+=\sfh_1 +\sum_{n\geq2,k>0}Q_{n,k}^+ t^k.\]
\end{definition}

Using Lemmas~\ref{lem:key1} and~\ref{lem:key2},
we first establish recursions for  
$Q_{n,k}$ and $Q_{n,k}^+$  \emph{without} auxiliary terms. Let
\[\bar Q:= \sum_{n\geq 2,k\geq0}Q_{n,k}t^k \and \bar Q^+ := \sum_{n\geq 2, k>0} Q_{n,k}^+ t^k.\]

\begin{proposition}\label{prop:main1}
	$\bar Q^+$ and $\bar Q$ satisfy 
	\[\bar Q^+=\sum_{\substack{a,r\geq0,\\ 0<b\leq a+r-2}}\sfh_a(\sfh_r\circ \bar Q^+)t^b \and \bar Q=\sum_{\substack{a,r\geq0, \\ 0\leq b\leq a+r-2}}\sfh_a(\sfh_r\circ \bar Q^+)t^b.\]	
\end{proposition}
\begin{proof}
	By Lemmas~\ref{lem:key1} and \ref{lem:key2}, $\bar Q^+$ is equal to the sum of
	\[\ch\left(U_{\Phi^+(T_\bullet)}\right)t^{k+b}=\sfh_at^b \prod_{j=1}^m\sfh_{r_j}\circ \left(\ch_{\symS_{n_j}}\left(U_{T_{j}}\right)t^{k_j}\right)\]
	for all possible combinations of
	\begin{itemize}
		\item integers $a,r\geq 0$ and $0<b\le a+r-2$;
		\item partitions $(r_1,\cdots, r_m)\vdash r$ with $m\geq 0$ and $r_j>0$;
		\item distinct weighted rooted trees $T_j\in \sT_{n_j,k_j}^+$,
	\end{itemize}
	where we set $m=0$ for $r=0$.
	In other words,
	 \[\bar Q^+=\sum_{\substack{a,r\geq0, \\ 0<b\leq a+r-2}}\sfh_at^b\sum_{\substack{(r_1,\cdots,r_m)\vdash r}}\;\sum_{\substack{T_j\in \sT^+ \text{distinct},\\ 1\leq j\leq m}}\;\prod_{j=1}^m\sfh_{r_j}\circ \left(\ch_{\symS_{n_j}}\left(U_{T_{j}}\right)t^{k_j}\right)\]
  where $T_j\in \sT_{n_j,k_j}^+$.
	 By Lemma~\ref{lem:plethysm.exp}, this is equal to
	 \[
	 	\sum_{\substack{a,r\geq0,\\ 0<b\leq a+r-2}}\sfh_at^b \left(  \sfh_r\circ \sum_{n,k\geq0}\sum_{T\in \sT_{n,k}^+}\ch_{\symS_n}\left(U_T\right) t^k\right)=\sum_{\substack{a,r\geq0,\\ 0<b\leq a+r-2}}\sfh_a \left(  \sfh_r\circ \bar Q^+\right)t^b.
	  \]
	 This completes a proof of the first formula. The same argument with $\Phi$ provides a proof of the second, which we omit.
\end{proof}

From Proposition~\ref{prop:main1} and the identities
\[Q=1+\sfh_1 + \bar Q \and Q^+=\sfh_1+\bar Q^+,\]
we deduce recursions for $Q^+$ and $Q$ as follows.

\begin{theorem}\label{thm:main1}
	$Q^+$ satisfies the following equivalent formulas
	\beq \label{eq:Q+}
	\begin{split}
		\mathrm{(Recursive)} \quad &Q^+=\sfh_1 +\sum_{r\geq 3}\left(\sum_{i=1}^{r-2}t^i\right)(\sfh_r\circ Q^+),\\
		\mathrm{(Exponential)} \quad &\Exp(tQ^+)=t^2\Exp(Q^+)+(1-t)(1+t+\sfh_1t ).
	\end{split}
		\eeq
 In particular, $\Exp(tQ^+)$ and $t^2\Exp(Q^+)$ are equal modulo low degree terms.

	Moreover, $Q$ is the plethystic exponential of $Q^+$:
	\beq \label{eq:Q}
	\begin{split}
		Q=\sum_{r\geq 0}\sfh_r\circ Q^+=\Exp(Q^+).
	\end{split}
	\eeq
	In particular, $\Log (Q)=Q^+$ satisfies the formulas in \eqref{eq:Q+}.
\end{theorem}
\begin{proof}
	By the first formula in Proposition~\ref{prop:main1}, $Q^+-\sfh_1 =\bar Q^+$ is equal to
	\beq\label{20}\begin{split}
		\sum_{\substack{a,r'\geq0,\\0<b\leq a+r'-2}}\sfh_a   \left(\sfh_{r'}\circ (Q^+-\sfh_1 )\right)t^b.
	\end{split}\eeq
	By Lemma~\ref{lem:plethysm.exp} and Example~\ref{ex:plethysm.-h1}, we can expand:
	\[\sfh_{r'}\circ(Q^+-\sfh_1 )=\sum_{r+c=r',\;r,c\geq0}\left(\sfh_r\circ Q^+\right)\cdot(-1)^c\sfe_c.\]
	Applying this to \eqref{20}, $Q^+-\sfh_1 $ is equal to
	\[\begin{split}
		&\sum_{\substack{a,c,r\geq0,\\ 0<b\leq a+c+r-2}}(-1)^c\sfh_a\sfe_c   \left(\sfh_r\circ Q^+\right)t^b=\sum_{\substack{d,r\geq0,\\ 0<b\leq d+r-2}}\left(\sfh_r\circ Q^+\right)t^b\sum_{\substack{a+c=d\\a,c\geq0}}(-1)^c\sfh_a\sfe_c.
	\end{split}\]
	Since $\sum_{c=0}^d(-1)^c\sfh_{d-c}\sfe_c=0$ for $d>0$ (cf. \cite[(I.2.6)]{Mac}), this is finally equal to the partial sum with $a=c=d=0$:
	\[\sum_{0<b\leq  r-2}\left(\sfh_r\circ Q^+\right)t^b=\sum_{r\geq 3}\left(\sum_{i=1}^{r-2}t^i\right)(\sfh_r\circ Q^+).\]
	The second formula in \eqref{eq:Q+} follows from the identity
	\beq\label{21}\begin{split}
		&\frac{t^2\Exp(Q^+)-\Exp(tQ^+)}{1-t}=\sum_{r\geq0}\frac{t^2-t^r}{1-t}\sfh_r\circ Q^+\\
		&\qquad =-(1+t)-tQ^++t\sum_{r\geq 3}\frac{t-t^{r-1}}{1-t}\sfh_r\circ Q^+=-(1+t+\sfh_1t ).
	\end{split}
	\eeq
	In fact, \eqref{21} proves the equivalence of the two formulas in \eqref{eq:Q+}.
	
	The formula \eqref{eq:Q} follows from the second formula in Proposition~\ref{prop:main1}:
	\[\begin{split}
		\bar Q&=\bar Q^+ +\sum_{\substack{a,r\geq 0\\a+r\geq 2}}\sfh_a(\sfh_r\circ \bar Q^+) =\sum_{\substack{a,r\geq0,\\(a,r)\neq (0,0),(1,0)}} \sfh_a(\sfh_r\circ \bar Q^+)\\
		&=-(1+\sfh_1 )+\sum_{a\geq0}\sfh_a\sum_{r\geq0} \sfh_r\circ\bar Q^+\\
		&=-(1+\sfh_1 )+\Exp(\sfh_1 )\Exp(\bar Q^+)=-(1+\sfh_1 )+\Exp(Q^+).
	\end{split}
	\]
	Hence, $Q=1+\sfh_1+\bar Q=\Exp(Q^+)$.
\end{proof}

\begin{example}
Using \eqref{eq:Q+} and \eqref{eq:Q}, we can recursively calculate a few low degree terms in $Q$. Recall that $\sfh_r\circ (t^kF) = t^{rk} (\sfh_r\circ F)$.

By \eqref{eq:Q+}, clearly we have $Q^+_1=\sfh_1$ and $Q^+_2=0$ and $Q^+$ starts as \[Q^+ = \sfh_1 + \cdots.\] Substituting this back to \eqref{eq:Q+}, we get 
\[\begin{split}  Q^+ &=\sfh_1 + t(\sfh_3\circ (\sfh_1 + \cdots )) \\
&=\sfh_1 + t\sfh_3\circ \sfh_1 + \cdots \\
&=\sfh_1 + t\sfh_3 + \cdots \\\end{split}
\] up to degree three. 
We repeat this process to get 
\[\begin{split}  Q^+ &=\sfh_1 + t(\sfh_3\circ (\sfh_1 +  \cdots )) +(t+t^2)(\sfh_4\circ (\sfh_1 +  \cdots ))\\
&=\sfh_1 + t\sfh_3 +(t+t^2) \sfh_4+\cdots   \end{split}
\] up to degree four 
and 
\[\begin{split}  Q^+ &=\sfh_1 + t(\sfh_3\circ (\sfh_1 +t\sfh_3 \cdots )) +(t+t^2)(\sfh_4\circ (\sfh_1 + \cdots ))+\\
&\hspace{2em}(t+t^2+t^3)(\sfh_5\circ (\sfh_1 + \cdots ))+\cdots \\
&=\sfh_1 + t\sfh_3 +(t+t^2) \sfh_4+ t(\sfh_2\circ \sfh_1)(\sfh_1\circ (t \sfh_3)) +(t+t^2+t^3)\sfh_5+\cdots \\
&=\sfh_1 + t\sfh_3 +(t+t^2) \sfh_4+ t^2 \sfh_{(3,2)}+ (t+t^2+t^3)\sfh_5+ \cdots
\end{split}
\] up to degree five. 
By taking the plethystic exponential, we get 
\[\begin{split}  Q &=1+\sfh_1 + t\sfh_3 +(t+t^2) \sfh_4+ t^2 \sfh_{(3,2)}+ (t+t^2+t^3)\sfh_5 +\\
& \hspace{2em}\sfh_2\circ (\sfh_1 + t\sfh_3  +(t+t^2) \sfh_4) +\sfh_3\circ (\sfh_1  + t\sfh_3)+\sfh_4\circ (\sfh_1)+\sfh_5\circ (\sfh_1 ) \cdots \\
&=1+\sfh_1 + t\sfh_3 +(t+t^2) \sfh_4+ t^2 \sfh_{(3,2)}+ (t+t^2+t^3)\sfh_5 +\\
& \hspace{2em}\sfh_2 + t\sfh_{(3,1)}+ (t+t^2)\sfh_{(4,1)} +\sfh_3+ t\sfh_{(3,2)}+\sfh_4+\sfh_5 \cdots\\
&=1+\sfh_1+ \sfh_2  + (1+t)\sfh_3 +(1+t+t^2) \sfh_4 + t\sfh_{(3,1)} + \\
& \hspace{2em} (1+t+t^2+t^3)\sfh_5 +(t+t^2)(\sfh_{(4,1)}+\sfh_{(3,2)})+  \cdots
\end{split}
\] up to degree five.
In the above computation, we repeatedly used Lemma \ref{lem:plethysm.exp} to expand the plethysm. Note that to compute $Q$ up to degree $n$, we need $Q^+$ up to degree $n$.
\end{example}

In general, we can derive inductive formulas for $Q_n^+$ and $Q_n$ with respect to $n$ as follows.

\begin{corollary}
\label{cor:main1}
	Set $Q_1^+=\sfh_1$. For $n\geq2$, $Q_n^+$ and $Q_n$ satisfy the following.
	\[\begin{split}
	    Q_n^+&=
	    \sum_{\lambda\vdash n}
	    \left(\sum_{i=1}^{\ell(\lambda)-2}t^i\right)\prod_{j=1}^m\left(\sfh_{r_j}\circ Q_{n_j}^+\right) \and
		Q_n=
		\sum_{\lambda\vdash n}
		\prod_{j=1}^m\left(\sfh_{r_j}\circ Q_{n_j}^+\right)
	\end{split}\]
	where $n_j$ denote the parts of $\lambda$ with multiplicities $r_j$ so that $\lambda=(n_1^{r_1},\cdots, n_m^{r_m})$ with $n_1>\cdots >n_m>0$
 and $\ell(\lambda)=\sum_{j=1}^mr_j$.
\end{corollary}

\begin{proof}
	These formulas 
    follow from the first formula in \eqref{eq:Q+} and the formula \eqref{eq:Q} respectively, by expanding $\sfh_r\circ Q^+=\sfh_r\circ(\sum_{n>0}Q_n^+)$
	using Lemma~\ref{lem:plethysm.exp},
	and considering the homogeneous pieces in $\Lambda_n\otimes_\Z \Z[t]$.
\end{proof}

\begin{remark}\label{rem:comp.rep}
Based on Corollary~\ref{cor:main1} and Theorem~\ref{thm:CKL}, we implemented the calculation of $P_n$ and $Q_n$ in a Mathematica program and computed them for $n\leq 25$. The program ran efficiently for values up to $n=20$, completing within an hour on a standard PC. Beyond this point, the computation time increases significantly, requiring several days to compute for $n=25$.
Based on this calculation, we verified Conjectures~\ref{conj:equiv.lc} and \ref{conj:mult.lc} for $n\leq 25$.
\end{remark}

\begin{example}[$k\leq 3$] One can also perform an induction on $k$, instead of $n$. 
	\[\begin{split}
		Q_{n,0}=&~\sfh_n \qquad (n\geq 2).\\
		Q_{n,1}=&~\sfh_n+\sum_{\substack{a+b=n\\ a\geq1,\;b\geq3}}\sfh_{(a,b)} \qquad (n\geq 3).\\
		Q_{n,2}=&~\sfh_n+\sum_{\substack{a+b=n\\a\geq1,\;b\geq4}}\sfh_{(a,b)}+\sum_{\substack{a+b=n\\ a\geq2,\;b\geq3}}\sfh_{(a,b)}+\sum_{\substack{a+b+c=n\\a\geq 3,\; b\geq 2,\;c\geq 1}}\sfh_{(a,b,c)}\\
		&+\sum_{\substack{a+2b=n\\ a\geq 3,\; b\geq0}}(\sfh_2\circ \sfh_a)\sfh_b+\sum_{\substack{a+b+c=n\\3\leq a<b,\; c\geq0}}\sfh_{(a,b,c)}  \qquad (n\geq 4).
	\end{split} \]
	The last two coincide with \cite[Example 5.14]{CKL}.
	
	Moreover, $Q_{n,3}$ is equal to
	\[\begin{split}
		&~\sfh_n+\sum_{\substack{a+b=n\\a\geq1,\; b\geq5}}\sfh_{(a,b)}+\sum_{\substack{a+b=n\\a\geq2,\;b\geq4}}\sfh_{(a,b)}+\sum_{\substack{a+b=n\\a\geq3,\;b\geq3}}\sfh_{(a,b)}+\sum_{\substack{a+b+c=n\\a,b\geq2,\;c\geq3}}\sfh_{(a,b,c)}\\
		&+\sum_{\substack{a+b+c=n\\a\geq1,\;b\geq2,\;c\geq 4}}\sfh_{(a,b,c)}+\sum_{\substack{a+b+c=n\\a,b\geq3,\; c\geq1}}\sfh_{(a,b,c)}+\sum_{\substack{a+b+c=n\\ a\geq3,\;b\geq4,\;c\geq0}}\sfh_{(a,b,c)}\\
		&+\sum_{\substack{2a+b=n\\a\geq3,\; b\geq 1}}(\sfh_2\circ \sfh_a)\sfh_{b}+\sum_{\substack{a+b+c=n\\3\leq a<b,\; c\geq 1}}\sfh_{(a,b,c)}+\sum_{\substack{a+b+c+d=n\\a\geq1,\;b,c\geq2,\;d\geq3}}\sfh_{(a,b,c,d)}\\
		&+\sum_{\substack{2a+b+c=n\\a\geq 3,\;b,c\geq 1}}(\sfh_2\circ \sfh_a)\sfh_{(b,c)}+\sum_{\substack{a+b+c+d=n\\3\leq a<b,\; c,d\geq 1}}\sfh_{(a,b,c,d)}+\sum_{\substack{a+b+c+d=n\\a\geq 2,\;b,c\geq 3,\;d\geq0}}\sfh_{(a,b,c,d)}\\
		&+\sum_{\substack{a+b+c+d=n\\3\leq a<b<c,\; d\geq0}}\sfh_{(a,b,c,d)}+\sum_{\substack{2a+b+c=n\\ a,b\geq3,\;a\neq b,\;c\geq0}}(\sfh_2\circ \sfh_a)\sfh_{(b,c)}+\sum_{\substack{3a+b=n\\a\geq3,\;b\geq0}}(\sfh_3\circ \sfh_a)\sfh_b,
	\end{split}\]
	for $n\geq 5$.
	
	Note that the formulas for $P_{n,k}$ with $k\leq 3$ are given in \cite[Corollary~6.2]{CKL}.
\end{example}

\begin{remark}\label{rem:Manin}
	Theorem~\ref{thm:main1} recovers Manin's characterization in \cite[Theorem~0.3.1]{Man} of the generating series of the Poincar\'e polynomial of $\Mbar_{0,n+1}$
	\[\varphi_n:=\sum_{k\geq0}^{n-2}\dim H^{2k}(\Mbar_{0,n+1})t^k ~\in \Z[t].\] Define its generating series as in \cite[\S0.3]{Man}:
	\[\varphi:=q+\sum_{n\geq2}\frac{\varphi_n}{n!}q^n. \]
	Let $\mathrm{rk}:\Lambda\llbracket t\rrbracket\to \Q\llbracket q,t\rrbracket$ be the \emph{rank homomorphism} defined in \cite[\S7.1]{GeKa} by
	\[\sfh_n\mapsto \frac{q^n}{n!}, \quad \text{or equivalently, }~\ch_{\symS_n}(V)\mapsto \frac{\dim (V) q^n}{n!} ~\text{ for }V\in R_n\]
	and $t^k\mapsto t^k$ for all $k$. Then it satisfies $\mathrm{rk}\circ \Exp(-)=\exp\circ\mathrm{rk}(-)$ (cf.~\cite[\S8.4]{GeKa}). Applying $\mathrm{rk}(-)$ to $Q$ and the formulas in Theorem~\ref{thm:main1}, one can immediately deduce that $\mathrm{rk}(Q)=1+\varphi$ and this is the unique solution to
	\beq\label{eq:exp.manin}\exp(t\log(1+\varphi))=t^2(1+\varphi) +(1-t)(1+t+qt)\eeq
	which is equivalent to
	\[(1+\varphi)^t=t^2\varphi +1+(1-t)qt,\]
	the equation written in \cite[(0.7)]{Man}. This shows that Theorem~\ref{thm:main1} is an equivariant generalization of Manin's characterization of $\varphi$.
	
	Similarly for the generating series $\chi:=\varphi(q,1)$ of the Euler characteristics of $\Mbar_{0,n+1}$, by differentiating \eqref{eq:exp.manin} by $t$ and specializing to $t=1$, we get
	\[(1+\chi)\log(1+\chi)=2\chi-q\]
	which is the equation written in \cite[(0.9)]{Man}.

\end{remark}

\begin{remark}
	There are dualities for $Q$ and $Q^+$:
	\[Q_{n,k}=Q_{n,n-2-k} \and Q^+_{n,k}=Q^+_{n,n-1-k}.\] Indeed, there is a bijection $\sT_{n,k}^+\cong \sT_{n,n-1-k}^+$, which sends $(T,w)$ to $(T,w^t)$ with $w^t(v):=val(v)-2-w(v)$ for every vertex $v$ of $T$. One can check that this is well defined. Similarly, we have $\sT_{n,k}\cong \sT_{n,n-2-k}$ (cf. \cite[Remark 5.13]{CKL}).
\end{remark}

\bigskip

\section{Multiplicities of the trivial representation}
\label{s:invariant}
As we have a recursive algorithm to compute the $\symS_n$-representations on the cohomology of $\Mbar_{0,n+1}$ and $\Mbar_{0,n}$, it is interesting to study the behavior of the multiplicities of each irreducible representation. In this section, we focus on the multiplicities of the trivial representation, which give the Poincar\'e polynomials of $\Mbar_{0,n+1}/\symS_n$ and $\Mbar_{0,n}/\symS_n$.

The key idea is to define an operation $\Inv$ that extracts the multiplicities of the trivial representations and to show that it commutes with the plethysm product (Proposition~\ref{prop:Inv.plethysm}). By applying this operation to the formulas in the previous section, we will establish a recursive algorithm to compute the Poincar\'e polynomials of $\Mbar_{0,n+1}/\symS_n$ and $\Mbar_{0,n}/\symS_n$.

\subsection{Projection to the invariant parts} \label{ss:Inv}
We define an operation which reads off the multiplicity of the trivial representation.
\begin{definition}
	For $n\geq 0$, define a projection map
	\[\Inv_n:\Lambda_n \lra \Z, \quad \sum_{\lambda\vdash n}c_\lambda \sfs_\lambda \mapsto c_{(n)}\]
	where we set $\sfs_{(0)}:=1$.
	This naturally extends to a (bigraded) map
	\[\Inv:\Lambda\llbracket t\rrbracket \lra \Z\llbracket q,t\rrbracket , \quad \sum_{n,k\geq0}F_{n,k}t^k\mapsto \sum_{n,k\geq0}\Inv_n(F_{n,k})q^nt^k \]
	where $F_{n,k}\in \Lambda_n$ and $q$ is a formal variable that records $n$.
\end{definition}
\begin{remark}\label{rem:Inv.rep}
	(1) For an $\symS_n$-representation $V$, $\Inv_n(\ch_{\symS_n}(V))=\dim V^{\symS_n}$. So, if $V$ is a transitive permutation representation, then $\Inv_n(\ch_{\symS_n}(V))=1$.
	
	(2) For every partition $\lambda$, $\Inv(\sfh_\lambda)=1$. In particular, 
 	\[\Inv \left(\sum_{n,k\geq0,\lambda\vdash m}d_{\lambda,k} \sfh_\lambda t^k \right)= \sum_{n,k\geq0,\lambda\vdash n}d_{\lambda,k} q^nt^k.\]

  (3) By (2), it is straightforward to check that $\Inv:\Lambda\llbracket t\rrbracket \lra \Z\llbracket q,t\rrbracket$ is the $\Z$-algebra homomorphism sending $t$ to $t$ and $\sfh_n$ to $q^n$.
\end{remark}

The following 
will be useful in the proof of Proposition~\ref{prop:Inv.plethysm} below.
\begin{lemma}\label{lem:Inv.plethysm}
	$\Inv(\sfh_r\circ \sfh_\lambda)=1$ for $r\geq0$ and partitions $\lambda\vdash n$ with $n\geq 1$.
\end{lemma}
\begin{proof}
	By \eqref{eq:ch.M.to.h}, Lemma~\ref{lem:plethysm.Ind} and Remark~\ref{rem:Inv.rep}, we have for $V=M^\lambda$,
	\[
	\Inv(\sfh_r\circ \sfh_\lambda)=\Inv\left(\sfh_r\circ \ch_{\symS_n}(V)\right)=\dim \left(\Ind^{\symS_{rn}}_{(\symS_n)^r\rtimes \symS_r}V^{\otimes r}\right)^{\symS_{rn}}.
	\]
Since $V$ is a transitive permutation representation of $\symS_n$, 
$\Ind^{\symS_{rn}}_{(\symS_n)^r\rtimes \sfh_r}V^{\otimes r}$ is a transitive permutation representation of $\symS_{rn}$. Hence $\Inv(\sfh_r\circ \sfh_\lambda)=1$.
\end{proof}

\begin{remark}\label{rmk:invhr}
	The same argument proves $\Inv(\sfh_r\circ \ch_{\symS_n}(V))=1$ for transitive permutation representations $V$. By applying it to the regular representation $V=M^{(1^n)}$, one can also deduce that
	\[\Inv(\sfh_r\circ \sfs_\lambda)=\begin{cases}
		1 &\text{if }\lambda=(n),\\
		0 &\text{otherwise}
	\end{cases}\]
	for Schur functions $\sfs_\lambda$.
\end{remark}

\subsection{Plethystic exponential}\label{ss:function.A}
The plethysm and the plethystic exponential defined in \S\ref{ss:plethysm} and \S\ref{ss:Exp} respectively are in fact well defined on other extended power series rings as well (cf.~\cite{Get2}). We consider these operations on $\Lambda\llbracket q,t\rrbracket$.

Define the plethysm $F\circ (-):\Z\llbracket q,t\rrbracket \to \Z\llbracket q,t \rrbracket$ for $F\in \Lambda$ as the operation uniquely determined by the following conditions
\begin{enumerate}
	\item $\sfp_n\circ f=f^{[n]}$;
	\item $(F+G)\circ f=F\circ f+G\circ f$;
	\item $(FG)\circ f=(F\circ f)(G\circ f)$
\end{enumerate}
for  $f\in \Z\llbracket q,t\rrbracket$, $n\geq 1$ and $F,G\in \Lambda$, where $f^{[n]}$ is defined as
\beq \label{eq:bracket.power}f^{[n]}(q,t)=f(q^n,t^n) \eeq
for $n\geq 1$ and we set $f^{[0]}:=1$.
More precisely, if $F=\sum_\lambda c_\lambda \sfp_\lambda$, then 
\begin{equation}\label{eq:pleth}
    F \circ f = \sum_\lambda c_\lambda f^{[\lambda]}
\end{equation}
where we write $f^{[\lambda]}:=f^{[\lambda_1]}\cdots f^{[\lambda_\ell]}$ for $\lambda=(\lambda_1,\cdots,\lambda_\ell)$.

We define the plethystic exponential as before. Let $\Z\llbracket q,t \rrbracket_+$ denote the subgroup of $\Z\llbracket q,t \rrbracket$ consisting of elements with no constant terms. Let
\begin{equation}\label{eq:PlethExp}
\Exp(f):=\exp\sum_{r\geq 1}\frac{\sfp_r}{r}\circ f =\exp\sum_{r\geq 1}\frac{f^{[r]}}{r}=\sum_{r\geq0}\sfh_r\circ f
\end{equation}
for $f\in \Z\llbracket q,t\rrbracket_+$. In particular, $\Exp(f)\in 1+ \Z\llbracket q,t \rrbracket_+$.
Note also that
\beq\label{eq:expmult} \Exp(f+g)= \Exp(f)\Exp(g), \eeq
as before and hence $\Exp(-f)=\Exp(f)^{-1}$.

\begin{example}[Monomials]\label{ex:A.monomial}
Let $a, n, k$ be integers with $n,k\geq0$. By \eqref{eq:PlethExp}, we have
\beq \label{eq:A.monomial}
\begin{split}
    \Exp(a q^nt^k) &= \exp\sum_{r\geq 1}\frac{a q^{rn}t^{rk}}{r} = \exp(-a \ln(1-q^nt^k))\\&=(1-q^nt^k)^{-a},
\end{split}  \eeq
provided that $(n,k)\ne (0,0)$.
Since $\Exp(a q^nt^k) =  \sum_{r\geq1} \sfh_r\circ(a q^nt^k)$, one can see that
\[\sfh_r\circ(a q^nt^k)= (-1)^r \binom{-a}{r}q^{rn}t^{rk},  \]
where $\binom{-a}{r}= (-1)^r\binom{a+r-1}{r}$ when $a>0$.
\end{example}

\begin{lemma}\label{lem:A.homomorphism}
	Let $f=\sum_{n,k}a_{n,k}q^nt^k$ with $a_{0,0}=0$. Then,
	\beq\label{eq:A.prod.form} \Exp(f)=\prod_{n,k\geq0, \; (n,k)\neq (0,0)}(1-q^nt^k)^{-a_{n,k}} .\eeq
\end{lemma}
\begin{proof}
The equality \eqref{eq:A.prod.form} holds if it holds when truncated up to $(q, t)$-degrees $(n, k)$ for arbitrary $n,k$. In particular, we may assume that $f$ is a finite sum, in which case, \eqref{eq:A.prod.form} follows from Example~\ref{ex:A.monomial} and \eqref{eq:expmult}.
\end{proof}

Now we are ready to prove that the map $\Inv$ commutes with $\Exp$.

\begin{proposition}\label{prop:Inv.plethysm}
The following diagram commutes.
\beq\label{eq:diag.Exp.A}\xymatrix{\Lambda\llbracket  t\rrbracket _+\ar[r]^-{\Exp}\ar[d]_-{\Inv}& 1+\Lambda\llbracket t\rrbracket _+\ar[d]^-{\Inv} \\ \Z\llbracket q,t\rrbracket _+\ar[r]^-{\Exp} &1+\Z\llbracket q,t\rrbracket _+ }\eeq
In other words, for every $F\in \Lambda\llbracket t\rrbracket _+$, we have
\beq \label{eq:Inv.plethysm}\Inv(\Exp(F))=\Exp(\Inv(F)).\eeq
In particular, $\Inv(\sfh_r\circ F)=\sfh_r\circ\Inv(F)$ for every $F\in \Lambda\llbracket  t\rrbracket$ and  $r\geq 0$.
\end{proposition}
\begin{proof}
By \eqref{eq:trunc} and truncation with respect to $(q,t)$-degree, we may assume that $F$ has finitely many summands of the form $\pm \sfh_\lambda t^k$ with $\lambda \vdash n$. By Remark~\ref{rem:Inv.rep} (3) and \eqref{eq:expmult}, we have
\[\begin{split}
\Inv(\Exp(F+G))&=\Inv(\Exp(F))\Inv(\Exp(G))\\
\Exp(\Inv(F+G))&=\Exp(\Inv(F))\Exp(\Inv(G)).
	\end{split}\]
Therefore, by induction on the number of summands, we may further assume $F= \pm \sfh_\lambda t^k$.

Moreover, if $F$ satisfies \eqref{eq:Inv.plethysm}, so does $-F$:
	\[\Inv(\Exp(-F))=\Inv(\Exp(F))^{-1}=\Exp(\Inv(F))^{-1}=\Exp(\Inv(-F)).\]

Hence, it suffices to show \eqref{eq:Inv.plethysm} for $F=\sfh_\lambda t^k$. In this case, $\Inv(\sfh_r\circ F)=q^{rn}t^{rk}$ for every $r$ by Lemma~\ref{lem:Inv.plethysm}. Consequently, $\Inv(\Exp(F))=\sum_{r\ge 0} q^{rn} t^{rk}=(1-q^nt^k)^{-1}$. This is equal to $\Exp(\Inv(F))=\Exp(q^nt^k)$, as computed in \eqref{eq:A.monomial}. This proves 
\eqref{eq:Inv.plethysm} and the proof for the last assertion is similar.
\end{proof}

\subsection{Multiplicities of the trivial representations}
We present a recursive algorithm to compute the multiplicities of the trivial representations by applying the map $\Inv$ to the formulas in the previous section.
\begin{definition}\label{def:fp.fq}
	Define
	\[\fp:=\Inv(P),  \quad \fq:=\Inv(Q) \and \fq^+:=\Inv(Q^+).\]
Let $\fp_n$, $\fq_n$ and $\fq_n^+\in \Z[t]$ be the coefficients of $q^n$, and let $\fp_{n,k}$, $\fq_{n,k}$ and $\fq_{n,k}^+\in \Z$ be the coefficients of $q^nt^k$ in the respective power series.
\end{definition}

Note that
\beq\label{eq:Betti.quotients}\fp_{n,k}=h^{2k}(\Mbar_{0,n}/\symS_n) \and \fq_{n,k}=h^{2k}(\Mbar_{0,n+1}/\symS_n)\eeq
are equal to the Betti numbers of the quotients $\Mbar_{0,n}/\symS_n$ and $\Mbar_{0,n+1}/\symS_n$ respectively. Thus, $\fp_n$ and $\fq_n$ are the Poincar\'e polynomials of the quotients $\Mbar_{0,n}/\symS_n$ and $\Mbar_{0,n+1}/\symS_n$ respectively.

\begin{theorem} \label{thm:main2}
	$\fq^+$ satisfies the following equivalent formulas:
	\[\begin{split}
		\mathrm{(Recursive)} \quad &\fq^+=q+\sum_{r\geq3}\left(\sum_{i=1}^{r-2}t^i\right)(\sfh_r\circ\fq^+),\\
		\mathrm{(Exponential)} \quad &\Exp(t\fq^+)=t^2\Exp(\fq^+)+(1-t)(1+t+qt).\\
	\end{split}\]
	Moreover, $\fq=\Exp(\fq^+)$.
\end{theorem}
\begin{proof}
These formulas immediately follow from those in Theorem~\ref{thm:main1}, due to Proposition~\ref{prop:Inv.plethysm}.
\end{proof}

\begin{corollary}
\label{cor:main2}
	Set $\fq_1^+=1$. For $n\geq2,$ $\fq_n^+$ and $\fq_n$ satisfy the following.
    \[\begin{split}
        \fq_n^+&=
        \sum_{\lambda\vdash n}
        \left(\sum_{i=1}^{\ell(\lambda)-2}t^i\right)\prod_{j=1}^m\left(\sfh_{r_j}\circ \fq_{n_j}^+\right)
        \and
		\fq_n=
		\sum_{\lambda\vdash n}
		\prod_{j=1}^m\left(\sfh_{r_j}\circ\fq_{n_j}^+\right)
    \end{split}\]
	where $n_j$ denote the parts of $\lambda$ with multiplicities $r_j$ so that
    $\lambda=(n_1^{r_1},\cdots, n_m^{r_m})$ with $n_1>\cdots >n_m>0$
	and $\ell(\lambda)=\sum_{j=1}^mr_j$.
\end{corollary}

\begin{proof}
	These follow from the two formulas 
 in Corollary~\ref{cor:main1} respectively by applying $\Inv$ to them. Indeed, by Remark~\ref{rem:Inv.rep} (3) and Proposition~\ref{prop:Inv.plethysm},
	$\Inv_n\left(\prod_{j=1}^m(\sfh_{r_j}\circ Q_{n_j}^+)\right)=\prod_{j=1}^m(\sfh_{r_j}\circ\fq_{n_j}^+)$. 
\end{proof}

We also obtain a formula relating $\fp$ and $\fq$ by applying the map $\Inv$ to the wall-crossing formula \eqref{eq:qwc.gen} in Corollary~\ref{cor:CKL.gen}.

\begin{theorem}
\label{thm:qwc.inv}
	$\fp$ and $\fq$ satisfy the following.
	\[(1+t)\fp=(1+t+qt)\fq-\frac{1}{2}t\left(\fq^2-\fq^{[2]}\right)\]
	where $\fq^{[2]}(q,t):=\sfp_2\circ \fq = \fq(q^2,t^2)$ as in \eqref{eq:bracket.power}.
\end{theorem}
\begin{proof}
	This follows from Corollary~\ref{cor:CKL.gen} by applying $\Inv(-)$ to it. Indeed,
	\[\Inv(\sfs_{(1,1)}\circ Q)=\Inv(Q)^2-\Inv(\sfh_2\circ Q)=\fq^2-\sfh_2\circ\fq=\frac{1}{2}\left(\fq^2-\fq^{[2]}\right)\]
	by Proposition~\ref{prop:Inv.plethysm}. The last equality follows from $\sfh_2= \frac12 \sfp_{(1,1)} + \frac12 \sfp_2 $.
\end{proof}

\medskip

\begin{corollary}\label{cor:qwc.inv} \cite[Corollary~6.6 (1)]{CKL}
	For $n\geq 3$,
	\beq\label{n1} (1+t)\fp_n=\fq_n-\frac{1}{2}t\left(\sum_{h=2}^{n-2}\fq_h\fq_{n-h}-\fq_{n/2}^{[2]}\right)\eeq
	where we set $\fq_{n/2}=0$ for odd $n$, and $\fq_{n/2}^{[2]}(t):=\sfp_2\circ\fq_{n/2}(t)=\fq_{n/2}(t^2)$.\footnote{There is a typo in \cite[Corollary~6.6 (1)]{CKL}, where $\fq_{n/2}$ should be $\fq_{n/2}^{[2]}$ as in \eqref{n1}. We thank Matthew Hase-Liu for bringing this to our attention.}
\end{corollary}
\begin{proof}
	Consider the coefficients of $q^n$  
	in Theorem~\ref{thm:qwc.inv}, or equivalently, apply $\Inv_n(-)$ to \eqref{eq:CKL}.
\end{proof}

\begin{remark}\label{rem:comp.inv}
	Based on Corollaries~\ref{cor:main2} and \ref{cor:qwc.inv}, we computed $\fp_n$ and $\fq_n$ for $n\leq 45$ using Mathematica. This computation took approximately 3 hours to complete for $n=45$ on a standard PC. With these results, we verified that Conjecture~\ref{conj} holds for $n\leq 45$.
\end{remark}

\bigskip

\section{Asymptotic log-concavity for the invariant part}
\label{s:asymptotic}
In this section, we prove asymptotic formuas for $\fp_n$ and $\fq_n$  which tell us that they are log-concave for sufficiently large $n$.
From the combinatorial formula for $Q_{n,k}$ in Proposition~\ref{prop:Q.wrt} and the wall-crossing formula in Theorem~\ref{thm:CKL} follow asymptotic formulas for $\fp_{n,k}$ and $\fq_{n,k}$ as functions of $n$ while keeping $k$ fixed. They exhibit polynomial growth of degree $k$, and we can explicitly calculate the leading coefficients, which allow us to verify the asymptotic log-concavity.

\subsection{Asymptotic formula for $\fq_{n,k}$}

By Proposition~\ref{prop:Q.wrt}, the $\symS_n$-module $H^{2k}(\Mbar_{0,n+1})$ is the direct sum of $U_{(T,w)}$ for each $(T,w)\in \sT_{n,k}$. Since each $U_{(T,w)}$ is a transitive permutation representation, one can immediately see that $\fq_{n,k}=\lvert \sT_{n,k}\rvert$. To study asymptotic behavior of $\fq_{n,k}$, we consider a stratification of $\sT_{n,k}$ based on the pairs of rooted trees and weight functions 
obtained by forgetting all inputs of elements in $\sT_{n,k}$.

Let $\overline{\sT}$ be the set of all pairs of rooted trees and weighted functions obtained from $\sT$
by removing all the inputs while keeping the weight function. Consider a function \[F:\sT\lra \overline{\sT}, \quad (T,w)\mapsto (F(T),w)\] which forgets all the inputs.
We denote by
\[F_{n}:~\bigsqcup_{k\geq0}\sT_{n,k}\hookrightarrow\sT\xrightarrow{~F~} \overline \sT\]
the restriction of $F$ to the set of weighted rooted trees with $n$ inputs.

Let $\overline \sT_k \subset \overline \sT$ denote the subset consisting of elements having weight $k$. In particular, $\overline\sT=\bigsqcup_{k\geq0}\overline\sT_k$ and $F_n(\sT_{n,k})\subset \overline\sT_k$.
Hence, the set $\sT_{n,k}$ admits a stratification \beq\label{eq:strat.Tnk}\sT_{n,k}=\bigsqcup_{T\in \overline{\sT}_k}F_{n}^{-1}(T).\eeq

\begin{definition}
	Let $n\geq2$ and $k\geq 0$.
	For $T\in \overline\sT_k$, define 
	\[\fq_{n,T}:=\lvert F_n^{-1}(T)\rvert.\]
	We will consider this as a function of $n$ for a given $T$.
\end{definition}
By definition, $\fq_{n,T}$ is 
the number of different ways of attaching $n$ inputs to (the vertices of) a given $T\in \overline \sT_k$ so that it becomes an element of $\sT_{n,k}$. For example, if $T\in \overline \sT_k$ has only one vertex (the root), then $\fq_{n,T}=1$ for $n\geq k+2$ and $\fq_{n,T}=0$ otherwise.

\smallskip
The stratification \eqref{eq:strat.Tnk} induces
\beq\label{eq:fq_k.into.fq_T}\fq_{n,k}=\sum_{T\in \overline{\sT}_k}\fq_{n,T}.\eeq

Let $T\in \overline \sT_k$. We denote by $\Aut(T)$ the group of automorphisms of the rooted tree $T$ without inputs. An automorphism of a rooted tree without inputs is by definition a permutation of non-root vertices sending edges to edges which preserves the weight function. If we assume that $\lvert V(T)\rvert=k+1$ so that every non-root vertex of $T$ has weight 1 and the root vertex of $T$ has weight 0, $\Aut(T)$ is a subgroup of $\symS_k$.
\begin{lemma}\label{lem:fq_T}
 Let $k\geq 0$ and $T\in \overline \sT_k$. Then, $\lvert V( T)\rvert\leq k+1$ and
	 \[\lim_{n\to \infty}\frac{\fq_{n,T}}{n^k}=\begin{cases}
	 	0 &\text{ if }~ \lvert V(T)\rvert \leq k\\
	 	\frac{1}{k!\cdot\lvert \Aut(T)\rvert} & \text{ if }~\lvert V(T)\rvert=k+1.
	 \end{cases}\]
\end{lemma}
\begin{proof}
	Let $m:=\lvert V(T)\rvert$. First note that we always have $m\leq k+1$, since every non-root vertex has a positive weight.
	
By definition, $\fq_{n,T}$ is less than or equal to the number of ways to write $n$ as a sum of $m$ non-negative integers, which is equal to $\binom{n+m-1}{m-1}$. Hence, when $m\leq k$, we have
	\[0\leq \lim_{n\to \infty}\frac{\fq_{n,T}}{n^k}\leq \lim_{n\to\infty}\frac{\binom{n+m-1}{m-1}}{n^k}=\lim_{n\to\infty}\frac{n^{m-1-k}}{(m-1)!}=0.\]
	
	Now suppose that $m=k+1$. Then every non-root vertex of $T$ has weight 1 and the root vertex $v_0$ of $T$ has weight 0.
To get an element in $\sT_{n,k}$, we attach $n$ inputs to $T$. Let $a_v$ denote the number of inputs to be attached to the vertex $v\in V(T)$ such that $n =\sum_v a_v$. Then the tree obtained after attaching inputs is an element of $\sT_{n, k}$ if and only if
\beq \label{eq:cond.inputs} a_v\geq 4-val_T(v) ~\text{ for }~ v\neq v_0 \and a_{v_0}\geq 3-val_T(v_0),\eeq
because the valency of $v$ of the new tree is equal to $a_v+val_T(v)$, where $val_T(v)$ denotes the valency of $v$ in $T$. Hence, when $\Aut(T)$ is trivial, $\fq_{n,T}$ is equal to the number of ways to write $n=\sum_{v\in V(T)}a_v$ with $k+1$ non-negative integers $a_v$ satisfying \eqref{eq:cond.inputs}. Clearly, if we let
	\[n_T:=n-\sum_{v\in V(T), v\neq v_0}\max\{4-val_T(v),0\}-\max\{3-val_T(v_0),0\},\]
this number is equal to
	\[\binom{n_T+k}{k}= \frac{1}{k!}n^k+o(n^k).\]
In general, when $\Aut(T)$ is not trivial, as the vertices of $T$ are unordered, the number should be divided by $\lvert\Aut(T)\rvert$.
\end{proof}

\begin{theorem}\label{thm:asymp.q_k}
	For $k\geq 0$, we have  
	\[\fq_{n,k}=\frac{(k+1)^{k-1}}{(k!)^2}n^k+o(n^k).\]
\end{theorem}
\begin{proof}
By Lemma~\ref{lem:fq_T} and \eqref{eq:fq_k.into.fq_T}, $\lim_{n\to \infty}\fq_{n,k}/(n^k/k!)$ is equal to the sum of $1/\lvert \Aut(T)\rvert$ over rooted trees $T\in \overline \sT_k$ with $k+1$ vertices. We claim that this sum multiplied by $k!$ is equal to the number of trees on $k+1$ labeled vertices, which is ${(k+1)^{k-1}}$, known as Cayley's tree formula.

We label non-root vertices of $T$ by integers from $1$ to $k$, while keeping the root labeled as $k+1$, so that we get a tree on $k+1$ labeled vertices. For each $T$, there are $k!$ ways to do that. Let the group $\mathbb{S}_k$ act on this labeled tree by permuting the labels. Then $\Aut(T)$ is exactly the stabilizing subgroup of the labeled tree. Clearly, this process is reversible and hence we have
\[\sum_{T\in \overline{\sT}_k}\frac{k!}{\lvert \Aut(T)\rvert} ={(k+1)^{k-1}}, \]
as required.
\end{proof}

\begin{corollary}\label{cor:asymp.q_k}
	The coefficients of the Poincar\'e polynomial of $\Mbar_{0,n+1}/\symS_n$ are asymptotically strictly log-concave, in the sense that for each $k\geq 1$,
	\[\fq_{n,k}^2>\fq_{n,k-1}\fq_{n,k+1}\]
	for any sufficiently large $n$.
\end{corollary}
\begin{proof}
	By Theorem~\ref{thm:asymp.q_k}, we have
	\[\lim_{n\to \infty}\frac{\fq_{n,k}^2}{\fq_{n,k-1}\fq_{n,k+1}}= \left(1+\frac{1}{k^2+2k}\right)^k\geq 1+\frac{1}{k+2}>1\]
	since $(1+x)^k\geq 1+kx$ for any $x\geq 0$.
\end{proof}

\begin{remark}\label{rem:c.exp}
	Let $c_k:=\frac{(k+1)^{k-1}}{k!}$. 
	Then its generating series
	\[\fc:=\sum_{k\geq0}c_kt^k=\sum_{k\geq0}\frac{(k+1)^{k-1}}{k!}t^k\]
	satisfies $\fc=\exp(t\fc)$. See \cite[Example 4i]{BEW}.
\end{remark}

\begin{remark}\label{rem:ACM}
	Our approach also allows us to recover \cite[Theorem~1.3]{ACM}:
	\[\lim_{n\to\infty}\frac{h^{2k}(\Mbar_{0,n+1})}{(k+1)^n}=c_k.\]
	Indeed, by arguments similar to those used in the proof of Lemma~\ref{lem:fq_T}, we find that for $T\in \bar \sT_k$,
	\beq\label{eq:nonequiv.limit}\lim_{n\to \infty}\frac{\sum_{T'\in F_n^{-1}(T)}\dim U_{T'}}{(k+1)^n}=\begin{cases}
		0 &\text{if }\lvert V(T)\rvert \leq k\\
		\frac{1}{\lvert\Aut(T)\rvert}&\text{if }\lvert V(T)\rvert=k+1.
	\end{cases}\eeq
	Hence, the sum of the left-hand side over all $T\in \bar \sT_k$ is equal to $c_k$.
\end{remark}

\subsection{Asymptotic formula for $\fp_{n,k}$}
Recall that
$$\fp_{n,k}=\Inv(P_{n,k})=h^{2k}(\Mbar_{0,n}/\symS_n).$$ We find an asymptotic formula for $\fp_{n,k}.$
We begin with explicit computation of $\fp_{n,k}$ with $k\leq 2$. 

\begin{example}\label{ex:fp_k.lowdeg}
For $0\leq k\leq 2$, the numbers $\fp_{n,k}$ can be explicitly computed using the formula for $P_{n,k}$ in \cite[Corollary 6.2]{CKL}. In particular,
	\[\begin{split}
		\fp_{n,0}&=1,\qquad \\
		\fp_{n,1}&=\left\lfloor\frac{n-2}{2}\right\rfloor=:m,\quad \text{ and}\\
		\fp_{n,2}&=\left\lfloor\frac{n-3}{2}\right\rfloor+\frac{1}{2}\left(\binom{n-4}{2}+\left\lfloor\frac{n-4}{2}\right\rfloor\right)=\begin{cases}
			m(m-1) &\text{ for }n \text{ even}\\
			m^2 & \text{ for } n \text{ odd},
		\end{cases}
	\end{split}
	\]
	where $m:=\lfloor\frac{n-2}{2}\rfloor$. Note that
	$\fp_{n,1}= \frac{n}{2}+o(n)$ and $\fp_{n,2}= \frac{n^2}{4}+o(n^2)$.
\end{example}

Now we prove that $\fp_{n,k}$ has polynomial growth of degree $k$ as $n$ increases, and compute its leading coefficient.
Write $c_k=\frac{(k+1)^{k-1}}{k!}$ as in Remark~\ref{rem:c.exp}, so that $\fq_{n,k}=\frac{c_k}{k!}n^k+o(n^k)$.

\begin{lemma}\label{lem:qwc.d_k}
	Let $k\geq 0$ and define
\[d_k :=c_k-\frac{1}{2}\sum_{j=0}^{k-1}c_jc_{k-1-j}.\]
Then we have
\[\fp_{n,k}=\frac{d_k}{k!}n^k+o(n^k).\]
\end{lemma}

\begin{proof}
Taking the degree $k$ parts of both sides of the formula in Corollary~\ref{cor:qwc.inv}, dividing by $n^k$ and taking their limits as $n\to \infty$,
we obtain
	\[\lim_{n\to\infty }\frac{\fp_{n,k}+\fp_{n,k-1}}{n^k}=\frac{c_k}{k!}-\frac{1}{2}\sum_{j=0}^{k-1}\frac{c_jc_{k-1-j}}{j!(k-1-j)!}\cdot\lim_{n\to\infty }\sum_{h=2}^{n-2}\frac{h^j(n-h)^{k-1-j}}{n^k}.\]
Since
	\[\lim_{n\to\infty}\sum_{h=2}^{n-2}h^a(n-h)^b\frac{1}{n^{k}}= \int_0^1x^a(1-x)^bdx=\frac{a!b!}{(a+b+1)!}\]
	for integers $a,b\geq0$,
	\[\lim_{n\to \infty}\frac{\fp_{n,k}+\fp_{n,k-1}}{n^k}=\frac{c_k}{k!}- \frac{1}{2}\sum_{j=0}^{k-1}\frac{c_jc_{k-1-j}}{k!}=\frac{d_k}{k!}.  \] 
	Now the assertion follows by an induction on $k$.
\end{proof}

Similarly as in $c_k$, the number $d_k$ can be explicitly calculated. 

\begin{theorem}\label{thm:asymp.p_k}
For $k\geq 0$, we have
\[ d_k=\frac{(k+1)^{k-2}}{k!}.\] In particular,
\[\fp_{n,k}= \frac{(k+1)^{k-2}}{(k!)^2} n^k+o(n^k).\]
\end{theorem}

\begin{proof}
    Let $\fc:=\sum_{k\geq0}c_kt^k$ and  $\fd:=\sum_{k\geq0}d_kt^k$. Then the assertion is equivalent to the identity $(t\fd)'=\fc$. By definition of $d_k$, we have $\fd = \fc - \frac{1}{2}t\fc^2$. By differentiating both sides of the identity $\exp(t\fc)=\fc$ discussed in Remark~\ref{rem:c.exp}, we obtain $\fc' = \fc^2 +t\fc\fc'$. Hence
    \[
        (t\fd)' = (t\fc - \frac{1}{2}t^2\fc^2)'=\fc +t\fc' -t\fc^2 -t^2 \fc\fc' =\fc,
    \]
    as required.
\end{proof}

\begin{corollary}\label{cor:asymp.p_k}
	The coefficients of the Poincar\'e polynomial of $\Mbar_{0,n}/\symS_n$ are asymptotically log-concave, in the sense that for each $k\geq 1$,
	\[\fp_{n,k}^2\geq \fp_{n,k-1}\fp_{n,k+1}\]
	for any sufficiently large $n$.
	Moreover, when $k>1$, the inequality is always strict as long as $n$ is sufficiently large. When $k=1$, the inequality holds for all $n$, and it is strict if and only if $n$ is even.
\end{corollary}
\begin{proof}
	First assume $k>1$.  By Theorem~\ref{thm:asymp.p_k}, we have
	\[\lim_{n\to\infty}\frac{\fp_{n,k}^2}{\fp_{n,k-1}\fp_{n,k+1}}= \left( 1+\frac{1}{k^2+2k}\right)^{k-1}\geq 1+\frac{k-1}{k^2+2k}> 1.\]
The assertion for $k=1$ follows from the explicit formulas for $\fp_k$ for $0\leq k\leq2$, stated in Example~\ref{ex:fp_k.lowdeg}.
\end{proof}

\subsection{Failure of ultra-log-concavity}\label{ss:ultra.lc}
	We say that a  sequence $a_0,\cdots, a_n$ of integers is \emph{ultra-log-concave} if the sequence $\left(a_i/\binom{n}{i}\right)_{0\leq i\leq n}$ is log-concave.

	In general, the coefficients of the Poincar\'e polynomials $\fp_n$ and $\fq_n$ of $\Mbar_{0,n}/\symS_n$ and $\Mbar_{0,n+1}/\symS_n$ do not satisfy the ultra-log-concavity: If  
	we let
	\[\tfp_{n,k}:=\frac{\fp_{n,k}}{\binom{n-3}{k}} \and \tfq_{n,k}:=\frac{\fq_{n,k}}{\binom{n-2}{k}},\]
	then the following holds.
\begin{corollary}\label{cor:ultra.lc}
	Let $k>0$. For any sufficiently large $n$, 
	\[\tfp_{n,k}^2<\tfp_{n,k-1}\tfp_{n,k+1} \and \tfq_{n,k}^2<\tfq_{n,k-1}\tfq_{n,k+1}.\]
\end{corollary}
\begin{proof}
	From 
	Theorem~\ref{thm:asymp.q_k}, we have 
	\[\begin{split}
		\lim_{n\to \infty}\frac{\tfq_{n,k}^2}{\tfq_{n,k-1}\tfq_{n,k+1}}&=\frac{(k+1)^{2k-1}}{k^{k-1}(k+2)^{k}}=\frac{k}{k+1}\left(1+\frac{1}{k^2+2k}\right)^{k}\\
		&<\frac{k}{k+1}\exp\left(\frac{1}{k+2}\right)<1,
	\end{split}
	\]
	where the last inequality holds by the well-known inequality $\exp(x)<\frac{1}{1-x}$ for $0<x<1$. Similarly, from Theorem~\ref{thm:asymp.p_k}, we have 
	\[\begin{split}
		\lim_{n\to \infty}\frac{\tfp_{n,k}^2}{\tfp_{n,k-1}\tfp_{n,k+1}}&=\frac{(k+1)^{2k-3}}{k^{k-2}(k+2)^{k-1}}=\frac{k}{k+1}\left(1+\frac{1}{k^2+2k}\right)^{k-1}\\
		&<\frac{k}{k+1}\exp\left(\frac{k-1}{k^2+2k}\right)<1
	\end{split}
	\]
	where the last inequality holds because $\exp\left(\frac{k-1}{k^2+2k}\right)<\frac{k^2+2k}{k^2+k+1}<\frac{k+1}{k}$.
\end{proof}

\subsection{Multiplicities of other irreducible representations}
We end this section with a discussion of the asymptotic behavior of the multiplicities of other irreducible representations in $P_{n,k}$ and $Q_{n,k}$. 
For $F\in \Lambda_n$ and $\lambda \vdash n$, we denote by $\mult_{\lambda}(F)$ the coefficient of $\sfs_\lambda$ in the Schur expansion of $F$. 

By Theorem~\ref{thm:asymp.q_k}, the functions $\fp_{n,k}$ and $\fq_{n,k}=\lvert \sT_{n,k}\rvert$ show polynomial growth as $n$ increases. Similar phenomena occur for $\mult_\lambda(P_{n,k})$ and $\mult_\lambda(Q_{n,k})$ as well. To state this precisely, we employ the following notations.

For $n\geq \lvert \lambda\rvert+\lambda_1$, we define the \emph{padded partition} (\cite[\S2.1]{CF})
\[\lambda[n]:=(n-\lvert\lambda\rvert,~\lambda_1,~\cdots,~\lambda_\ell).\]
For example, when $\lambda=(1)$, $\lambda[n]=(n-1,1)$ is the partition corresponding to the standard representation of $\symS_n$. For a partition $\lambda:=(\lambda_1,\cdots, \lambda_\ell)$, we write $\lvert \lambda\rvert:=\sum_{i}\lambda_i$ and $\lambda!:=\prod_{i=1}^\ell \lambda_i!$.

\begin{theorem}\label{thm:bound}
	Let $\lambda$ be a partition, and let $k\geq0$. Then the functions $n\mapsto \mult_{\lambda[n]}(P_{n,k})$ and  $n\mapsto \mult_{\lambda[n]}(Q_{n,k})$ are bounded by polynomials of degree $\lvert \lambda\rvert+k$. Moreover,
	\[\limsup_{n\to\infty}\frac{\mult_{\lambda[n]}(P_{n,k})}{n^{\lvert \lambda \rvert +k}}\leq \limsup_{n\to\infty}\frac{\mult_{\lambda[n]}(Q_{n,k})}{n^{\lvert \lambda \rvert +k}}\leq \frac{(k+1)^{k-1}}{\lambda! (k!)^2}.\]
\end{theorem}
\begin{proof}
It suffices to prove the assertions on $Q_{n,k}$, because $Q_{n,k}-P_{n,k}$ is Schur-positive. Indeed, since the forgetful morphism $\Mbar_{0,n+1}\to \Mbar_{0,n}$ is $\symS_n$-equivariant and admits a section, the pullback map $H^{2k}(\Mbar_{0,n})\to H^{2k}(\Mbar_{0,n+1})$ is $\symS_n$-equivariant and injective.

By Proposition~\ref{prop:Q.wrt}, $\mult_{\lambda[n]}(Q_{n,k})$ is the sum of $\mult_{\lambda[n]}(U_T)$ over all $T\in \sT_{n,k}$. Since $U_T$ is a subrepresentation of the regular representation, $\mult_{\lambda[n]}(U_T)\leq \mult_{\lambda[n]}(M^{(1^n)})=\dim S^{\lambda[n]}$ (see \eqref{eq:reg.rep}). Thus, \[\mult_{\lambda[n]}(Q_{n,k})\leq \dim S^{\lambda[n]}\cdot \lvert \sT_{n,k}\rvert.\]

    Since $S^{\lambda[n]}\subset M^{\lambda[n]}$ and the dimension of $M^{\lambda[n]}$ is equal to the multinomial coefficient
	$\frac{n!}{\lambda[n]!}=\frac{n^{\lvert \lambda \rvert }}{\lambda!}+o(n^{\lvert\lambda\rvert})$, we have
	$\limsup_{n\to \infty}\frac{\dim S^{\lambda[n]}}{n^{\lvert\lambda\rvert}}\leq \frac{1}{\lambda!}$. Hence, the assertions immediately follow from Theorem~\ref{thm:asymp.q_k}.
\end{proof}

\begin{remark}\label{rem:bound.not.sharp}
	The degree $\lvert \lambda \rvert+k$ of the polynomial bounds in the above theorem is not sharp. For example, let $\lambda=(a)$ so that $\lambda[n]=(n-a,a)$. When $k=1$, we have
	\[\mult_{\lambda[n]}(Q_{n,1})=n-a-\max\{a,3\}+1=n+o(n)\]
	which is bounded by a linear polynomial,
	while $\lvert \lambda \rvert =a$ can be arbitrarily large. Similarly, one can check that $\mult_{\lambda[n]}(P_{n,1})=\frac{n}{2}+o(n)$.
\end{remark}

\begin{remark}
There is another interesting fact that can be derived from the combinatorial formula \eqref{eq:Qnk}. Namely, $\sfs_\lambda$ with $\ell(\lambda)> k+1$ or $\lambda_1<3$ does not appear in the Schur expansion of $Q_{n,k}$ (and hence in that of $P_{n,k}$). This is essentially because a weighted rooted tree $T\in \sT_{n,k}$ has at most $k+1$ vertices where inputs can be attached, and it has at least one vertex $v$ with at least three inputs.

We do not include a full proof of this fact, as a more general statement can be found in \cite[Theorem~4.2]{FP} and \cite[Theorem~5.1 and Corollary~5.2]{BM}.
\end{remark}

\end{document}